\documentclass[reqno]{amsart}
\usepackage[a4paper]{geometry}
\usepackage[utf8]{inputenc} 
\usepackage[english]{babel}
\usepackage{amsmath,amsfonts,amssymb,amscd,bm, amsthm}
\usepackage{latexsym}
\usepackage{graphicx}
\usepackage[margin=1cm]{caption}
\usepackage{subcaption}
\usepackage{tikz}
\usepackage{pgfplots}
\pgfplotsset{compat=1.16}

\usepackage{fancyhdr}
\usepackage{color}
\usepackage{listings}
\usepackage{pdfpages}
\usepackage{hyperref}
\usepackage{hhline}
\usepackage{courier}
\usepackage{mathrsfs}
\usepackage{wrapfig}
\usepackage{chngcntr}
\usepackage{multicol}
\usepackage{xcolor,colortbl}
\usepackage{bigfoot}
\usepackage{afterpage}
\usepackage{xurl}
\usepackage{tabularx,colortbl}

\usepackage[backend=bibtex, style=numeric-comp]{biblatex}
\appto{\bibsetup}{\sloppy}

\newtheorem{theorem}{Theorem}[section]
\newtheorem{lemma}[theorem]{Lemma}

\theoremstyle{definition}

\newtheorem{assumption}[theorem]{Assumption}

\theoremstyle{remark}

\newcommand{\R}{\mathbb{R} }
\newcommand{\C}{\mathbb{C} }

\newcommand{\dd}{\,\mathrm{d} }

\numberwithin{equation}{section}

\begin{document}

\title{Weak coupling asymptotics for the Pauli operator in two dimensions}

\author{Matthias Baur}
\address{Institute of Analysis, Dynamics and Modeling, Department of Mathematics, University of Stuttgart, Pfaffenwaldring 57, 70569 Stuttgart, Germany}
\email{matthias.baur@mathematik.uni-stuttgart.de}

\begin{abstract}
We compute asymptotic expansions for the negative eigenvalues of the Pauli operator in two dimensions perturbed by a weakly coupled potential with definite sign. Whereas previous results were limited to the case of radial magnetic fields and potentials, we are able to drop the assumption of radial symmetry entirely.
\end{abstract}

\maketitle


\section{Introduction and Main results}

\subsection{Introduction}

Given a magnetic vector potential $A\in L^2_{loc}(\R^2;\R^2)$ and its associated magnetic field $B=\text{curl }A $, we consider the two-dimensional Pauli operator
\begin{align}
P(A) = \begin{pmatrix}
P_+(A) & 0 \\
0 & P_-(A) 
\end{pmatrix}, \qquad P_\pm(A) = (i\nabla +A)^2 \pm B, \nonumber
\end{align}
acting on $L^2(\R^2 ; \C^2)$. It models a spin-$\frac{1}{2}$ fermion interacting with a magnetic field perpendicular to the plane and is obtained as the non-relativistic limit of the Dirac operator, see Thaller \cite{Thaller1992} for more details. The operators $P_\pm(A)$ on the diagonal are the spin-up and spin-down components of the Pauli operator and defined via closure of the quadratic forms
\begin{align}
\int_{\R^2} |(i\nabla +A)u|^2 \pm B|u|^2 \dd x, \qquad u\in C_0^\infty(\R^2). \nonumber
\end{align}
Under suitable decay and regularity conditions on $B$, the Pauli operator is essentially self-adjoint on $C_0^\infty(\R^2; \mathbb{C}^2)$ and $\sigma(P(A)) = \sigma_{\text{ess}}(P(A)) = [0,\infty)$, see for example Cycon et al.~\cite{Cycon1987} and Avramska-Lukarska et al.~\cite{AvramskaLukarska2023}. The spectrum for each of the spin-components $P_\pm(A)$ is also $[0,\infty)$.

For $B\in L^1(\R^2)$, let
\begin{align}
\alpha := \frac{1}{2\pi} \int_{\R^2} B(x) \dd x < \infty \nonumber
\end{align}
be the normalized magnetic flux. The Aharonov-Casher theorem \cite{Cycon1987, Aharonov1979, Erdoes2002} states that $P(A)$ has a zero eigenvalue if $|\alpha| > 1$. Its multiplicity is
\begin{align} \label{eq:n_aharonov_casher_states}
n = \# \{ k \in \mathbb{N}_0 \,: \, k< |\alpha|- 1 \} =  \begin{cases}
\lfloor |\alpha| \rfloor, & \alpha \in \R \setminus \mathbb{Z}, \\
|\alpha|- 1, & \alpha \in\mathbb{Z}\setminus \{0 \}.
\end{cases}
\end{align}
Here, $\lfloor \, . \, \rfloor$ denotes the floor function. If $\alpha > 1$, then the zero eigenvalues originate purely from the spin-down component $P_-(A)$ while the spin-up component $P_+(A)$ does not exhibit a zero eigenvalue. The opposite holds if $\alpha < -1$. In this case, $P_+(A)$ has a zero eigenvalue of multiplicity $n$ while $P_-(A)$ has none. The zero eigenstates of $P(A)$ are commonly called Aharonov-Casher states.

In the following, we consider perturbations of the Pauli operator of the form
\begin{align}
H(\varepsilon) = P(A) - \varepsilon \mathcal{V},  \qquad \varepsilon>0, \nonumber
\end{align}
where 
\begin{align}
 \mathcal{V} = \begin{pmatrix}
V_1 & 0 \\
0 & V_2
\end{pmatrix} \nonumber
\end{align}
and $V_1$, $V_2$ denote multiplication operators with real-valued potentials that are suitably regular and fast decaying. As usual in the literature, we denote the potentials also with $V_1$, $V_2$. For easier presentation, we will restrict our discussion to the case $V_1 =V_2 =V$. Physically, this case corresponds to a perturbation with a small electric field. Note however that since the analysis that follows treats both diagonal components of $H(\varepsilon)$ separately, it is straightforward to state our results also for two potentials $V_1$ and $V_2$ that do not coincide. 

For a wide class of potentials, the essential spectrum of the perturbed Pauli operator $H(\varepsilon)$ remains that of $P(A)$, see \cite{Cycon1987}. If the perturbation is attractive, then one expects that negative eigenvalues emerge from the bottom of the essential spectrum of the unperturbed Pauli operator, since for sufficiently small coupling parameter $\varepsilon$, the interaction between the potential and the zero eigenstates pushes the zero eigenvalues down. In the case where $\varepsilon$ is small, the potential $V$ is usually called ``weakly coupled''.

Weakly coupled potentials are indeed physically relevant. The interaction between an electron's spin and a magnetic field is characterized by the gyromagnetic ratio of the magnetic moment $g$. While $g=2$ for particles described by the unperturbed Pauli operator, the experimentally measured gyromagnetic ratio of electrons exhibits an anomaly which slightly shifts the $g$-factor to $g \approx 2.0023$. This shift can be explained by QED corrections. In the 1990s, various authors observed that the presence of a magnetic field together with any anomalous magnetic moment $g>2$ allows binding of electrons. The anomalous magnetic moment of electrons was taken into account by adding small perturbations $\varepsilon V_{1,2} = \pm\frac{1}{2}(g-2) B$ to the Pauli operator. We briefly review selected results.

In 1993, Bordag and Voropaev \cite{Bordag1993} established the existence of $n+1$ bound states for $\alpha \in \mathbb{R} \setminus \mathbb{Z}$ in three explicitly solvable models. Also relying on an explicit model, Cavalcanti, Fraga and de Carvalho \cite{Cavalcanti1997} later discussed the case of a magnetic field that is constant on a disc and zero outside. The first step towards general magnetic field profiles was done by Bentosela, Exner and Zagrebnov \cite{Bentosela1998}, who showed that magnetic fields with a rotational symmetry admit at least one bound state if the field is strong enough. Then, Cavalcanti and de Carvalho \cite{Cavalcanti1998} were able to construct a suitable set of test functions to show the existence of at least
\begin{align}
n' = \begin{cases}
n+1, & \alpha \in \R \setminus \mathbb{Z}, \\
n+2, & \alpha \in\mathbb{Z},
\end{cases} \nonumber
\end{align} 
bound states with negative energy, assuming also rotational symmetry of the magnetic field and non-negative sign. Symmetry assumptions and assumptions on the sign of the magnetic field could be dropped in the following, see \cite{Bentosela1999, Bentosela1999a}. 

Observe that the number of negative eigenvalues $n'$ under weak perturbations is strictly larger than $n$, the number of zero eigenvalues of the unperturbed Pauli operator. The difference between $n'$ and $n$ is caused by so-called virtual bound states at zero or zero resonant states of $P(A)$. If $\alpha \geq 0$, $P_-(A)$ exhibits one or two virtual bound states at zero. These are zero modes of $P_-(A)$ that are in $L^\infty(\R^2) \setminus L^2(\R^2)$. The number of virtual bound states at zero of $P_-(A)$ depends on whether the magnetic flux $\alpha$ is integer or non-integer. Similarly, if $\alpha \leq 0$, then $P_+(A)$ shows one or two virtual bound states at zero. Each virtual bound state at zero leads to one additional negative eigenvalue for the weakly coupled Pauli operator.

While the previously mentioned papers predominantly discussed the existence of bound states for certain weakly perturbed Pauli operators, it is natural to ask for approximate expressions for the bound state energies for a wider class of weakly coupled potentials. In this paper, we derive asymptotic expansions of the negative eigenvalues of $H(\varepsilon)$ in the limit $\varepsilon \searrow 0$, the weak coupling limit. 

For Schrödinger operators $-\Delta - \varepsilon V$, asymptotics for eigenvalues in the weak coupling limit were already discussed by several authors in the 1970s, see e.g.\ \cite{Reed1978, Simon1976, Simon1977, Klaus1980, Blankenbecler1977, Klaus1977}. It was discovered that although the free Laplacian does not have a zero eigenvalue, Schrödinger operators in dimensions one and two exhibit one negative eigenvalue in the weak coupling limit, provided that the added potential is attractive. More precisely, in two dimensions, Simon \cite{Simon1976} showed in 1976 that if $\int V \dd x >0$, then for small enough $\varepsilon$, there exists one negative eigenvalue $\lambda(\varepsilon)$ with 
\begin{align}
\lambda(\varepsilon) \sim - \exp\left(-4\pi \left(  \int_{\R^2} V(x) \dd x  \right)^{-1} \varepsilon^{-1} \right), \qquad \varepsilon \searrow 0. \label{eq:simon_asymptotic_proof}
\end{align}
His calculations are based on the Birman-Schwinger principle and the representation of the Birman-Schwinger operator as an integral operator using the integral kernel of $(-\Delta - \lambda)^{-1}$. Following Simon's approach, Arazy and Zelenko \cite{Arazy2005,Arazy2005a} proved asymptotic expansions for negative eigenvalues of generalized Schrödinger operators $(-\Delta)^l - \varepsilon V$ in $\R^d$ for $2l \geq d$. Fractional Schrödinger operators have also been discussed \cite{Hatzinikitas2010}.

Extraction of weak coupling asymptotics via Birman-Schwinger operators and explicit resolvent expressions could not be applied in the same way for the Pauli operator, since its resolvent is not as easily expressed by an integral kernel as the resolvent of the free Laplacian. This can be seen as a reason why methods applied in the setting of anomalous magnetic moments, see again \cite{Cavalcanti1997, Cavalcanti1998, Bordag1993, Bentosela1998, Bentosela1999, Bentosela1999a}, to prove existence of negative eigenvalues instead relied on rather explicit models and variational principles. 

Using the abstract Birman-Schwinger principle, Weidl \cite{Weidl1999} was able to compute the number of negative eigenvalues of $H(\varepsilon)$ for small enough $\varepsilon$ when $B$ is bounded, compactly supported and the potential $V$ is sufficiently regular. For non-negative potential $V$, his result guarantees the existence of exactly $n'$ bound states in the weak coupling limit. However, his approach did not yield asymptotic expansions for the eigenvalues.

Frank, Morozov and Vugalter \cite{Frank2010} were able to compute weak coupling asymptotics for the Pauli operator by further pushing the variational approach. They worked however under the rather restrictive assumption of radial, compactly supported $B$ and radial, non-negative $V$. Here, the assumption of radial fields allowed decomposing the perturbed Pauli operator into several half-line operators which were subsequently treated variationally.

We will generalize the weak coupling asymptotics computed by Frank, Morozov and Vugalter to $B$ and $V$ that do not necessarily exhibit a radial symmetry. In contrast to Frank, Morozov and Vugalter, we return to the Birman-Schwinger principle approach and make use of asymptotic expansions of resolvents of the Pauli operator found in a recent paper by Kova{\v{r}}{\'{\i}}k \cite{Kovarik2022}. We extract the eigenvalue asymptotics by iterated use of the Schur-Livsic-Feshbach-Grushin (SLFG) formula, see Lemma \ref{lem:slfg_formula}. We note that Section 10 of \cite{Kovarik2022} already contains a brief analysis of weak coupling asymptotics in the case $0<\alpha < 1$. The main purpose of this paper is to extend this analysis to any magnetic flux $\alpha$.

The paper is organized as follows: In the following, we will recall the notion of Aharonov-Casher states and state our main results. Asymptotic expansions for negative eigenvalues of $H(\varepsilon)$ are given for three mutually exclusive cases: the zero-flux case, the non-integer flux case and the non-zero, integer flux case. In Section \ref{sec:prelim_not}, we will set up more notation and list the resolvent expansions on which our calculations are based. Section \ref{sec:weak_coupling_proof} will be concerned with the proofs of the main results. We give a proof for each of the three cases mentioned above.

\subsection{Gauge and Aharonov-Casher states}

The zero modes of the Pauli operator, the Aharonov-Casher states, play a central role in the following. To define them, we first need to set a gauge for the magnetic field $B$. Although the results presented here do not depend on the specific choice of gauge, we choose the gauge as in \cite{Kovarik2022}, so that the results therein hold. Let 
\begin{align}
h(x) = - \frac{1}{2\pi} \int_{\R^2} B(y) \log|x-y| \dd y. \nonumber
\end{align}
We set the canonical vector potential
\begin{align}
A_h(x) = (\partial_{2}h(x), -\partial_1 h(x)). \nonumber
\end{align}
Then, $\operatorname{curl} A_h = -\Delta h = B$, so $A_h$ induces $B$. Furthermore, note that
\begin{align}\label{eq:h_alpha_log}
h(x) = -\alpha \log |x| + O(|x|^{-1}), \qquad |x| \to \infty,  
\end{align}
which implies $e^{\mp h}=O(|x|^{\pm\alpha})$ as $|x| \to  \infty$. This property of $h$ will become important shortly. Finally, we make a gauge transformation and replace $A_h$ by $A = A_h + \nabla \chi$, where $\chi$ is the transformation function $\chi: \R^2 \to \R$ constructed in \cite[Sec. 3.1]{Kovarik2022}.

We now construct the Aharonov-Casher states, the zero eigenstates of $P(A)$. For this, let $N_\pm$ be the space of zero eigenstates of $P_\pm(A)$ (in other words, the kernel of $P_\pm(A)$). It is easily verified that for $v\in C_0^\infty(\R^2)$
\begin{align}
\int_{\R^2} |(i\nabla +A_h)(e^{\mp h} v)|^2 \pm B|e^{\mp h} v|^2 \dd x = \int_{\R^2} e^{\mp 2 h} |(\partial_1 \mp i \partial_2 )v|^2 \dd x \nonumber
\end{align}
and it follows that any zero mode of $P_\pm(A)=P_\pm(A_h + \nabla \chi)$, i.e.\ a solution of the equation $P_\pm(A)u=0$, must have the form $u=e^{\mp h + i \chi} v$ with $v$ analytic in $x_1\pm i x_2$. 

Suppose $\alpha > 0$. A zero mode $u$ is a zero eigenstate if $u \in L^2(\R^2)$. Requiring $u \in L^2(\R^2)$ forces $v$ to be a polynomial of finite degree. To see this, it is enough to realize that if $v$ is a polynomial of degree $m$, then
\begin{align}
e^{\mp h} v = |x|^{m \pm \alpha} (1+o(1)), \qquad |x | \to \infty, \nonumber
\end{align}
due to \eqref{eq:h_alpha_log}, so the degree must satisfy $m< \mp \alpha -1$. The spin-up component $P_+(A)$ has therefore a trivial zero eigenspace, i.e.\
\begin{align}
N_+ = \{ 0 \}, \nonumber 
\end{align}
while the spin-down component $P_-(A)$ has the zero eigenspace
\begin{align}
N_- &= \operatorname{span}\{e^{  h(x) + i\chi(x) },(x_1 - i x_2)e^{  h(x) + i\chi(x) }, ...,  (x_1 - i x_2)^{n-1} e^{  h(x) + i\chi(x) }\} \nonumber
\end{align}
with $n$ from \eqref{eq:n_aharonov_casher_states}. We note that under enough regularity of $B$, the function $e^h$ is non-zero almost everywhere. Hence, the states $(x_1 - i x_2)^k e^{  h(x) + i\chi(x) }$ are linearly independent and hence indeed $\operatorname{dim}(N_-)=n$. Finally, the zero eigenstates of the full Pauli operator $P(A)$ are given by
\begin{align} \label{eq:full_pauli_zero_eigenstates}
\begin{pmatrix}
\psi^+ \\
\psi^- 
\end{pmatrix},\quad \psi^\pm \in N_\pm.
\end{align} 
This is the statement of the Aharonov-Casher Theorem.

Virtual bound states at zero are zero modes of $P(A)$ that are bounded but not in $L^2(\R^2;\mathbb{C}^2)$. To define them, we need to find all bounded zero modes of $P_\pm(A)$ first. Hence, let $N_\pm^\infty$ denote the space of bounded zero modes of $P_\pm(A)$. By the same arguments as above, we see that for $u$ to be bounded, the function $v$ must be a polynomial of degree $m\leq \mp \alpha$. For $\alpha >0$, this implies that the spin-up component $P_+(A)$ does not have any bounded zero modes, while the spin-down component $P_-(A)$ admits, in addition to its zero eigenstates, the bounded zero modes
$$(x_1 - i x_2)^{n} e^{  h(x) + i\chi(x) }$$
if $\alpha \in \mathbb{R} \setminus \mathbb{Z}$ or 
$$(x_1 - i x_2)^{n} e^{  h(x) + i\chi(x) }, \qquad (x_1 - i x_2)^{n+1} e^{  h(x) + i\chi(x) }$$
if $\alpha \in \mathbb{Z}$. Therefore,
\begin{align}
N_+^\infty &= \{ 0 \} , \nonumber \\
N_-^\infty &=\begin{cases} \operatorname{span}\{e^{  h(x) + i\chi(x) },(x_1 - i x_2)e^{  h(x) + i\chi(x) }, ...,  (x_1 - i x_2)^{n} e^{  h(x) + i\chi(x) }\}, & \alpha \in\mathbb{R} \setminus \mathbb{Z}, \\
\operatorname{span}\{e^{  h(x) + i\chi(x) },(x_1 - i x_2)e^{  h(x) + i\chi(x) }, ...,  (x_1 - i x_2)^{n+1} e^{  h(x) + i\chi(x) }\}, & \alpha \in \mathbb{Z}.
\end{cases} \nonumber
\end{align}

The situation is analogous in the case $\alpha <0$. Here, $P_-(A)$ has trivial zero eigenspace and no bounded zero modes, i.e.~$N_-=N_-^\infty = \{ 0 \}$, while $P_+(A)$ has the $n$-dimensional zero eigenspace 
\begin{align}
N_+ &= \operatorname{span}\{e^{ - h(x) + i\chi(x) },(x_1 + i x_2)e^{ - h(x) + i\chi(x) }, ...,  (x_1 + i x_2)^{n-1} e^{ - h(x) + i\chi(x) }\}  \nonumber
\end{align}
and the space of bounded zero modes
\begin{align}
N_+^\infty &=\begin{cases} \operatorname{span}\{e^{ - h(x) + i\chi(x) },(x_1 + i x_2)e^{  -h(x) + i\chi(x) }, ...,  (x_1 + i x_2)^{n} e^{  -h(x) + i\chi(x) }\}, & \alpha \in\mathbb{R} \setminus \mathbb{Z}, \\
\operatorname{span}\{e^{  -h(x) + i\chi(x) },(x_1 + i x_2)e^{  -h(x) + i\chi(x) }, ...,  (x_1 + i x_2)^{n+1} e^{  -h(x) + i\chi(x) }\}, & \alpha \in \mathbb{Z}. 
\end{cases} \nonumber
\end{align}

A special case is the case $\alpha = 0$. In this case, the operators $P_\pm(A)$ both have no zero eigenstates, but both exhibit a bounded zero mode, given by $e^{  \mp h(x) + i\chi } $. Hence, 
\begin{align}
N_\pm = \{ 0 \}, \qquad N_\pm^\infty &= \operatorname{span}\{e^{ \mp h(x) + i\chi(x) }\}.  \nonumber
\end{align}

Having determined the spaces of $L^2$-integrable and bounded zero modes of $P_\pm(A)$, virtual bound states at zero of $P_\pm(A)$ are simply all states in $(N_+^\infty \setminus N_+) \cup \{ 0\}$ resp. $(N_-^\infty \setminus N_-) \cup \{ 0\}$. As before, the bounded zero modes and virtual bound states at zero of the full Pauli operator are attained by composing the respective states $ \psi^\pm \in (N_\pm^\infty \setminus N_\pm) \cup \{ 0\}$ as in \eqref{eq:full_pauli_zero_eigenstates}.

In the following, we will often work either with the spin-up component $P_+(A)$ or the spin-down component $P_-(A)$.  We will often use the term ``Aharonov-Casher states`` synonymously for the above constructed zero eigenstates of the spin components $P_\pm(A)$, i.e.\ the functions
\begin{align}
\psi_k^\pm(x) = (x_1\pm ix_2)^{k} e^{\mp h(x)+i\chi(x)},  \label{eq:aharonov_casher_states}
\end{align}
where $k=0,...,n_\pm-1$, $n_\pm = \operatorname{dim}(N_\pm)$. Additionally, we will use the term ``generalized Aharonov-Casher state`` for the states $\psi_k^\pm$ with $k=n$ and $k=n+1$ when they are virtual bound states.

\subsection{Main results}

The weak coupling asymptotics presented in the following require assumptions on the regularity and decay behaviour of the magnetic field $B$ and the potential $V$. 

For the magnetic field, we assume the following.

\begin{assumption} \label{assump:magnetic_field}
The magnetic field $B:\R^2 \to \R$ is continuous and satisfies
\begin{align}
|B(x)| \lesssim (1+|x|^2)^{-\rho} \nonumber
\end{align}
for some $\rho > 7/2$.
\end{assumption}

This assumption is a sufficient condition for the validity of the resolvent expansions of the Pauli operator recalled in Section \ref{subsec:resolvent_expansions} that we apply in our proofs.

Additionally, for the potential, we make the following assumption.

\begin{assumption} \label{assump:potentialV}
The potential $V:\R^2 \to \R$ satisfies $V \geq 0$, $V> 0$ on a set of positive measure and
\begin{align}
V(x) \lesssim (1+|x|^2)^{-\sigma}  \nonumber
\end{align}
for some $\sigma > 3$.
\end{assumption}

Note that this assumption implies $V \in L^1(\R^2)$ and $\int_{\R^2} V(x) \dd x > 0$. 

Moreover, we can assume that the magnetic flux satisfies $\alpha \geq 0$. This is because $P_\pm(- A)$ is unitarily equivalent to $P_\mp(A)$, so flipping the sign of $B$ and hence of $\alpha$ is equivalent to exchanging the roles of $P_+(A)$ and $P_-(A)$. The corresponding results for $\alpha<0$ thus follow similarly to the case $\alpha>0$.

We present the asymptotics in three mutually exclusive cases: $\alpha = 0$, $\alpha \in\R \setminus \mathbb{Z}$ and $\alpha \in\mathbb{Z}\setminus \{ 0\}$. We begin with $\alpha = 0$.

\begin{theorem}[zero flux] \label{thm:expansion_zero}
Let $B:\R^2 \to \R$ satisfy Assumption \ref{assump:magnetic_field} and $V:\R^2 \to \R$ satisfy Assumption \ref{assump:potentialV}. Assume $\alpha=0$. Then for all sufficiently small $\varepsilon > 0$, the operator $H(\varepsilon)$ has precisely two negative eigenvalues $\lambda_0^\pm(\varepsilon)$ which satisfy
\begin{align}
\lambda_0^\pm(\varepsilon) = -\exp\left( -  \mu_{\pm} ^{-1} \, \varepsilon^{-1} (1+O(\varepsilon)) \right) \label{eq:thm_zero_flux_asymp}
\end{align}
as $\varepsilon \searrow 0$ with
\begin{align}
\mu_{\pm} = \frac{1}{4\pi}\int_{\R^2} V|\psi_0^\pm |^2 \dd x. 
\end{align}
\end{theorem}

Observe that if $B\equiv 0$, then the free Pauli operator acts as two copies of the free Laplacian $-\Delta$. In that case holds $\psi_0^\pm  \equiv 1$ and \eqref{eq:thm_zero_flux_asymp} becomes just two copies of Simon's weak coupling asymptotic \eqref{eq:simon_asymptotic_proof}.

When the magnetic flux is positive, i.e.\ $\alpha>0$, there exists at least one virtual bound state at zero, possibly a collection of zero eigenstates and a second linear independent virtual bound state at zero if $\alpha$ is an integer. 
We denote by $P_0^-$ the projector onto the zero eigenspace of $P_-(A)$. If the zero eigenspace is non-trivial, then $V^\frac{1}{2} P_0^- V^\frac{1}{2}$ is a non-trivial, bounded, self-adjoint and non-negative linear operator on $L^2(\R^2)$ of finite rank. It will be shown in Section \ref{sec:weak_coupling_proof} that Assumption \ref{assump:potentialV} implies that the rank of $V^\frac{1}{2} P_0^- V^\frac{1}{2}$ is equal to $n$, the dimension of the zero eigenspace of $P_-(A)$, see \eqref{eq:n_aharonov_casher_states}, and $V^\frac{1}{2} P_0^- V^\frac{1}{2}$ has $n$ positive eigenvalues (counted with multiplicities). Let $Q$ denote the orthogonal projection onto $(\operatorname{ran}(V^\frac{1}{2} P_0^- V^\frac{1}{2}))^\perp$, the orthogonal complement of the range of $V^\frac{1}{2} P_0^- V^\frac{1}{2}$. Also, let $\varphi^- \in N_-^\infty \setminus N_-$ be the particular virtual bound state at zero that appears in the resolvent expansion given in Theorem \ref{thm:resolvent_non_int} (see \cite[Sec. 5]{Kovarik2022} for the construction). With these definitions, we can state the weak coupling asymptotics for $\alpha >0$.

For positive, non-integer flux, we have the following theorem.

\begin{theorem}[non-integer flux] \label{thm:expansion_nonint}
Let $B:\R^2 \to \R$ satisfy Assumption \ref{assump:magnetic_field} and $V:\R^2 \to \R$ satisfy Assumption \ref{assump:potentialV}. Assume $\alpha >0$, $\alpha \in \mathbb{R} \setminus \mathbb{Z}$, set $n= \lfloor \alpha \rfloor$ and $\alpha'= \alpha-  \lfloor \alpha \rfloor$. Then for all sufficiently small $\varepsilon > 0$, the operator $H(\varepsilon)$ has precisely $n+1$ negative eigenvalues $\lambda_0(\varepsilon)$, ..., $\lambda_{n}(\varepsilon)$ which satisfy
\begin{align}
\lambda_k(\varepsilon) &= - \mu_k \, \varepsilon \, (1 + O(\varepsilon^{\min\{ \alpha',1-\alpha'\}} )),  \qquad  k=0, ..., n-1, \\
\lambda_{n}(\varepsilon) &=  - \mu_n \, \varepsilon^\frac{1}{\alpha'} \, (1+ O(\varepsilon^{\min\{1,\frac{1}{\alpha'}-1\}} )) \label{eq:expansion_nonint_lambda_n}
\end{align}
as $\varepsilon \searrow 0$. Here, $\{\mu_k \}_{k=0}^{n-1}$ are the positive eigenvalues of $V^\frac{1}{2} P_0^- V^\frac{1}{2}$ and $\mu_n$ is given by
\begin{align}
\mu_n &= \left( \frac{4^{\alpha'-1} \Gamma(\alpha') }{\pi \Gamma(1-\alpha')}\, \Vert Q (V^\frac{1}{2} \varphi^-) \Vert_{L^2(\R^2)}^{2} \right)^\frac{1}{\alpha'}. 
\end{align}
\end{theorem}

Note that if $0< \alpha < 1$, then $\alpha'=\alpha$, $n=0$ and the Pauli operator has trivial eigenspace, i.e.\ $P_0^- = 0$ and $\varphi^- = \psi_0^-$. In that case, the second order term of $\lambda_n(\varepsilon)= \lambda_0(\varepsilon)$ can be improved. The asymptotic expansion of $\lambda_0(\varepsilon)$ is then given by
\begin{align}
\lambda_0(\varepsilon) = - \left( \frac{4^{\alpha'-1} \Gamma(\alpha') }{\pi \Gamma(1-\alpha')}\, \int_{\R^2} V |\psi_0^-|^2 \dd x \right)^\frac{1}{\alpha'}  \varepsilon^\frac{1}{\alpha'}  \, (1+ O(\varepsilon ))  \label{eq:expansion_alphaless1}
\end{align}
as $\varepsilon \searrow 0$.

If the flux is positive and integer, let $\varphi_1^-$, $\varphi_2^- \in N_-^\infty \setminus N_-$ be the two particular virtual bound states at zero that appear in the resolvent expansion given in Theorem \ref{thm:resolvent_int} (see \cite[Sec. 6]{Kovarik2022} for the construction) and let $Q$ be again the orthogonal projection onto $(\operatorname{ran}(V^\frac{1}{2} P_0^- V^\frac{1}{2}))^\perp$. Moreover, let $\tilde{Q}$ be the orthogonal projection onto $(\operatorname{ran}(V^\frac{1}{2} P_0^- V^\frac{1}{2})+ \operatorname{span}\{ V^\frac{1}{2}\varphi_2^- \})^\perp$ and let $K$ be the linear operator also defined in Theorem \ref{thm:resolvent_int}. Then, we have the next theorem. 

\begin{theorem}[integer flux] \label{thm:expansion_int}
Let $B:\R^2 \to \R$ satisfy Assumption \ref{assump:magnetic_field} and $V:\R^2 \to \R$ satisfy Assumption \ref{assump:potentialV}. Assume $\alpha>0$, $\alpha \in \mathbb{Z}$ and set $n = \alpha -1$. Then for all sufficiently small $\varepsilon > 0$, the operator $H(\varepsilon)$ has precisely $n+2$ negative eigenvalues $\lambda_0(\varepsilon)$, ..., $\lambda_{n+1}(\varepsilon)$ which satisfy
\begin{align}
\lambda_k(\varepsilon) &= - \mu_k \, \varepsilon  \left(1+ O\left(|\log \varepsilon|^{-1} \right) \right),  \qquad k=0, ...,n-1, \\
\lambda_{n}(\varepsilon) &=  - \mu_n \frac{\varepsilon}{|\log \varepsilon|} \left(1 + O\left( \frac{\log|\log \varepsilon|}{|\log \varepsilon|} \right) \right) , \\
\lambda_{n+1}(\varepsilon) &= - \exp\left( -\mu_{n+1}^{-1}  \varepsilon^{-1} \left(1 + O\left(\varepsilon
\right) \right) \right)  
\end{align}
as $\varepsilon \searrow 0$. Here, $\{ \mu_k \}_{k=0}^{n-1}$ are given by the non-zero eigenvalues of $V^\frac{1}{2} P_0^- V^\frac{1}{2}$ and $\mu_n$, $\mu_{n+1}$ are given by
\begin{align}
\mu_n &= \frac{1 }{\pi }\, \Vert \phi_2 \Vert_{L^2(\R^2)}^{2}, \qquad \phi_2 =  Q  (V^\frac{1}{2} \varphi_2^-), \label{eq:mu_n_thm_int}\\
\mu_{n+1} &= \frac{\langle \phi_1 , V^\frac{1}{2} K V^\frac{1}{2}  \, \phi_1 \rangle}{\Vert \phi_1 \Vert_{L^2(\R^2)}^{2} }, \qquad \phi_1 =  \tilde{Q}  (V^\frac{1}{2} \varphi_1^-).\label{eq:mu_nplus1_thm_int}
\end{align}
\end{theorem}

The proofs of Theorems \ref{thm:expansion_zero}, \ref{thm:expansion_nonint} and \ref{thm:expansion_int} are given in Section \ref{sec:weak_coupling_proof}. We conclude this section with several remarks. 

In the case of radial $B$ and $V$, our asymptotic expansions reproduce those found by Frank, Morozov and Vugalter in \cite{Frank2010}, where the coefficients $\{\mu_k \}$ are given more explicitly in terms of the (generalized) Aharonov-Casher states $\{ \psi_k^\pm \}$ and the second order error terms are slightly improved. This can be explained by observing that the (generalized) Aharonov-Casher states are pairwise orthogonal in case of radial fields. The missing interaction between the (generalized) Aharonov-Casher states eventually leads to simpler expressions for the coefficients $\{\mu_k \}$ and less pollution of the eigenvalue asymptotics in second order. We refer the reader to the appendix for a detailed discussion of the case of radial fields.

The asymptotic expansion \eqref{eq:expansion_alphaless1} of $\lambda_0(\varepsilon)$ for $0<\alpha<1$ coincides with that found by Frank, Morozov and Vugalter for radial fields, but \eqref{eq:expansion_alphaless1} holds for general, potentially non-radial $B$ and $V$. The improvement of the second order error term from $O(\varepsilon^{\min\{1,\frac{1}{\alpha'}-1\}} )$ in \eqref{eq:expansion_nonint_lambda_n} to $O(\varepsilon)$ in \eqref{eq:expansion_alphaless1} is again due to absence of interaction of (generalized) Aharonov-Casher states. In this case however, the absence of interaction between (generalized) Aharonov-Casher states has the trivial reason that there is only a single generalized Aharonov-Casher state.

We expect that the Assumptions \ref{assump:magnetic_field} and \ref{assump:potentialV} can be weakened. The stated regularity and decay behaviour of the magnetic field $B$ are sufficient conditions for the resolvent expansions of the Pauli operator we cite. It is expected that the resolvent expansions hold under weaker assumptions on $B$ (although $B$ should at least be in $L^1(\R^2)$). Such weaker assumptions would then also be applicable here. The decay behaviour of the potential $V$ we assume is also dictated by the resolvent expansions. The core issue here is that the resolvent expansions are given in terms of bounded operators between certain weighted $L^2$-spaces and not the usual $L^2(\R^2)$ space. Since our proofs rely on the Birman-Schwinger principle, a certain decay of $V$ is required to be able to carry over these resolvent expansions to the Birman-Schwinger operators $V^{\frac{1}{2}} (P_\pm(A)-\lambda)^{-1} V^{\frac{1}{2}}$ acting on $L^2(\R^2)$.

Moreover, we have assumed that the potential $V$ is point-wise non-negative and fulfills $\int V \dd x > 0$. For a sign-indefinite potential, the problem of computing weak coupling eigenvalue asymptotics remains open. One major difficulty comes from the fact that the number of negative eigenvalues that appear under weak coupling is then not necessarily equal to the number of Aharonov-Casher plus generalized Aharonov-Casher states, but instead given by the number of non-negative eigenvalues of the matrix
$\hat{V} = \left(\int V \overline{\psi_k} \psi_l \dd x \right)_{k,l=0}^{n \text{ or } n+1}$ where $\{ \psi_k \}_{k=0}^{n \text{ or } n+1}$ are the (generalized) Aharonov-Casher states, see Weidl \cite{Weidl1999}. We think that if $\hat{V}$ has only positive eigenvalues, the assumption that $V$ is non-negative is not necessary and similar asymptotic expansions should continue to hold for such sign-indefinite potentials $V$.

Clearly, our results must cease to hold as soon as $\hat{V}$ exhibits a negative eigenvalue. When the sign of $V$ is negative (i.e.\ a repulsive potential is coupled), recent results by Breuer and Kova{\v{r}}{\'{\i}}k \cite{Breuer2023} show the appearance of resonances. In the generic sign-indefinite case, we expect that $H(\varepsilon)$ exhibits a mix of bound states as well as resonances with asymptotic expansions that not necessarily exhibit the same leading order behaviour as derived here or in \cite{Breuer2023}.

\section{Preliminaries and Notation} \label{sec:prelim_not}

\subsection{Notation} \label{subsec:notation}

Let $X$ be a set and let $f_1, f_2: X \to \R$ be two functions. We write $f_1(x) \lesssim f_2(x)$ if there exists a constant $c>0$ such that $f_1(x) \leq c f_2(x)$ for all $x\in X$. We write $f_1(x) \asymp f_2(x)$ if $f_1(x) \lesssim f_2(x)$ and $f_2(x) \lesssim f_1(x)$.

For two Hilbert spaces $H$, $H'$ let $\mathcal{L}(H, H')$ denote the Banach space of bounded linear operators from $H$ to $H'$. For brevity we also write $\mathcal{L}(H)$ for $\mathcal{L}(H, H)$.

For a subspace $U\subset H$, let $U^\perp = H \ominus U$ denote its orthogonal complement and let $Q_U$ be the orthogonal projector onto $U$.

For a linear operator $A\in\mathcal{L}(H)$ and two subspaces $U,V \subset H$ we define $A|_{U\to V} \in \mathcal{L}(U,V)$ by $A|_{U\to V} : U \to V$, $x \mapsto Q_V Ax$. This operator is the component of $A$ that acts on $U$ and maps into $V$.

In Section \ref{subsec:resolvent_expansions}, we will recall resolvent expansions of the operators $P_\pm(A)$ given in terms of operators acting on weighted $L^2$-spaces. For this purpose, let $L^{2,s}(\R^2)$, $s\in\mathbb{R}$, denote the weighted $L^2$-space equipped with the norm
\begin{align}
\Vert  u \Vert_{2,s} = \Vert (1+|\cdot |^2)^\frac{s}{2} u \Vert_{L^2(\R^2)}. \nonumber
\end{align}
We denote by $\mathcal{B}(s,s')$ the space of bounded linear operators from $L^{2,s}(\R^2)$ to $L^{2,s'}(\R^2)$. Furthermore, for $s\in\mathbb{R}$ and $u\in L^{2,s}(\R^2)$, $v\in L^{2,-s}(\R^2)$, let $u \langle v, \, \cdot \,\rangle $ denote the linear operator in $\mathcal{B}(s,-s)$ that acts as
\begin{align}
u \langle v, f\rangle := u \int_{\R^2}\overline{v} f \dd x. \nonumber
\end{align}

\subsection{Resolvent expansions} \label{subsec:resolvent_expansions}

At the core of the following proofs are resolvent expansions for $P_\pm(A)$ derived in \cite{Kovarik2022}. There are three cases that need to be distinguished. The first case is $\alpha = 0$, where the resolvents of $P_\pm(A)$ look similar for the $+$ and $-$ case.

\begin{theorem}[Cor. 7.6 \cite{Kovarik2022}]\label{thm:resolvent_int_zero}
Let $\alpha=0$. Then
\begin{align}
(P_{\pm}(A)-\lambda)^{-1}= - \frac{\log  \lambda }{4 \pi}  \psi_0^\pm \langle \psi_0^\pm, \, . \,  \rangle   +O(1) \nonumber
\end{align}
holds in $\mathcal{B}(s,-s)$, $s>3$, as $\lambda \rightarrow 0$, $\operatorname{Im}(\lambda)\geq 0$. 
\end{theorem}

The branch of the complex logarithm is chosen such that $\log  \lambda = \log  |\lambda| + i \pi$ for $\lambda< 0$. Therefore, one can also write
\begin{align}
(P_{\pm}(A)-\lambda)^{-1}= - \frac{\log  |\lambda |}{4 \pi}  \psi_0^\pm \langle \psi_0^\pm, \, . \,  \rangle   +O(1) \nonumber
\end{align}
as $\lambda \rightarrow 0$, instead.

If the magnetic flux is positive, the situation is different for $P_+(A)$ and $P_-(A)$. For the resolvent of the spin-up component we have the following simple expansion.

\begin{theorem}[Prop. 5.11, Prop. 6.8 \cite{Kovarik2022}]\label{thm:resolvent_plus}
Let $\alpha>0$. Then
\begin{align}
(P_{+}(A)-\lambda)^{-1}= O(1) \nonumber
\end{align}
holds in $\mathcal{B}(s,-s)$, $s>3$, as $\lambda \rightarrow 0$, $\operatorname{Im}(\lambda)\geq 0$. 
\end{theorem}

For the other component, $P_-(A)$, the situation is more complicated. It requires us to adopt a bit of notation of \cite{Kovarik2022} to state the resolvent expansions. 

For $0<t\in \R \setminus \mathbb{Z}$ let us define
\begin{align} \label{eq:zeta_omega_definition}
\zeta(t) = - \frac{4^{t-1} \Gamma(t) e^{i\pi t}}{\pi \Gamma(1-t)}, \qquad \omega(t) = \frac{|d_n^-|^2}{\zeta(t) }
\end{align}
with $d_n^- \in \mathbb{C}$. The precise definition of $d_n^-$ is found in \cite{Kovarik2022}, but it is not relevant for the results nor the proofs presented in this paper. Furthermore, let $P_0^-$ denote the orthogonal projector onto the zero eigenspace of $P_-(A)$, let $\psi^-\in N_-$ denote a particular zero eigenfunction of $P_-(A)$ and $\varphi^- \in N_-^\infty \setminus N_-$ a particular virtual bound state at zero. For the precise definition of the states $\psi^-$ and $\varphi^-$ we refer to \cite[Sec. 5]{Kovarik2022}. We emphasize that they are certain linear combinations of the (generalized) Aharonov-Casher states $\{\psi_k^-\}_{k=0}^n$.

Then, the following asymptotic expansion for the resolvent is valid for non-integer magnetic flux.

\begin{theorem}[Thm. 5.6 \cite{Kovarik2022}] \label{thm:resolvent_non_int}
Let $0<\alpha \in \R \setminus \mathbb{Z}$ and $\alpha'=\alpha- \lfloor \alpha \rfloor$. Then there are constants $\tau, \rho \in\R$ such that
\begin{align}
(P_-(A)-\lambda)^{-1}= -\lambda^{-1} P_0^- + \frac{\omega(1+\alpha') \lambda^{\alpha'-1}}{1+\tau \omega(1+\alpha') \lambda^{\alpha'}} \psi^- \langle \psi^-, . \rangle - \frac{\zeta(\alpha') \lambda^{-\alpha'}}{1+\rho \zeta(\alpha') \lambda^{1-\alpha'}} \varphi^- \langle \varphi^-, . \rangle  + O(1)  \nonumber
\end{align}
holds in $\mathcal{B}(s,-s)$, $s>3$, $\lambda \rightarrow 0$, $\operatorname{Im}(\lambda)\geq 0$. If $\alpha<1$, then the zero eigenspace of $P_-(A)$ is trivial and the asymptotic expansion holds with $P_0^- = 0$ and $\psi^-=0$.
\end{theorem}

Definitions of the constants $\tau, \rho$ can also be found in \cite{Kovarik2022}, however their precise values are not relevant for our proof.

At integer magnetic flux, one constructs two virtual bound states $\varphi_1^-$ and $\varphi_2^-$ by linear combination of (generalized) Aharonov-Casher states $\{\psi_k^-\}_{k=0}^{n+1}$ such that $\varphi_1^- \in L^\infty(\R^2) \setminus L^p(\R^2)$ for any $2 \leq p < \infty$ and $\varphi_2^- \in L^p(\R^2) \setminus L^2(\R^2)$ for any $2 < p \leq \infty$. For the precise construction, we refer again to \cite[Sec. 6]{Kovarik2022}. Given these virtual bound states $\varphi_1^-$ and $\varphi_2^-$, let
\begin{align}
\Pi_{jk} = \varphi_k^- \langle \varphi_j^-, \, . \, \rangle , \qquad j,k=1,2. \label{eq:definition_Pi_op}
\end{align}
The resolvent expansion in the integer flux case then takes the following form.

\begin{theorem}[Thm. 6.5 \cite{Kovarik2022}]\label{thm:resolvent_int}
Let $\alpha \in \mathbb{Z}$, $\alpha>0$. Then there are constants $m\in\R$, $\kappa\in\C$ such that
\begin{align}
(P_-(A)-\lambda)^{-1}= -\lambda^{-1} P_0^- + \frac{\Pi_{22}}{\pi \lambda (\log \lambda +m- i\pi)} - K (\log \lambda -i \pi) + O(1) \nonumber
\end{align}
holds in $\mathcal{B}(s,-s)$, $s>3$, as $\lambda \rightarrow 0$, $\operatorname{Im}(\lambda)\geq 0$. Here,
\begin{align}
K = \frac{1}{4\pi} \left( \Pi_{11} + \overline{\kappa} \Pi_{12} + \kappa \Pi_{21} + |\kappa|^2 \Pi_{22} \right) + \frac{\pi |d_n^-|^2}{4 } \psi^- \langle \psi^- , \, . \, \rangle . \label{eq:definition_K_op}
\end{align}
If $\alpha = 1$, then the above expansion holds with $P_0^-=0$ and $\psi^- = 0$.
\end{theorem}

We remark that there appears a typographical error in \cite{Kovarik2022} concerning the case $\alpha = 1$ where Theorem 6.5 reads ``$P_0^-  = \Pi_{22} = 0$'' instead of ``$P_0^-=\psi^- = 0$''. Again, the precise definitions of the constants $d_n^-$, $m$ and $\kappa$ can be found in \cite{Kovarik2022}.

\subsection{Schur complement}

Another important ingredient to our calculations is the Schur-Livsic-Feshbach-Grushin (SLFG) formula and in particular the Schur complement of a bounded, self-adjoint operator given in block form. We will apply the following version of the SLFG formula.

Let $H$ be a Hilbert space and suppose $P$ is an orthogonal projector. Let $H_1 = P H$ and $H_2 =(1 - P) H $. Then $H_1$, $H_2$ are closed subspaces of $H$ and $H=H_1 \oplus H_2$. Let $A$ be a bounded, self-adjoint operator on $H$. We write $A$ in block form
\begin{align}
A = \begin{pmatrix}
A_{11} & A_{12} \\
A_{21} & A_{22}
\end{pmatrix} \nonumber
\end{align}
where $A_{11}= A|_{H_1\to H_1}$, $A_{12}= A|_{H_2\to H_1}$, $A_{21}= A|_{H_1\to H_2}$ and $A_{22}= A|_{H_2\to H_2}$ (recall the notation introduced in the beginning of this section). 
\begin{lemma}[SLFG, \cite{Sjoestrand2007, Zhang2005, Jensen2001}] \label{lem:slfg_formula}
Assume $A_{22}$ is invertible on $H_2$. Then $A$ is invertible on $H$ if and only if its Schur complement $S=A_{11} - A_{12} A_{22}^{-1} A_{21}$ is invertible on $H_1$. In particular, 
\begin{align} \label{eq:slfg_kernel_dim}
\operatorname{dim} \operatorname{ker} A = \operatorname{dim} \operatorname{ker} S. 
\end{align}
\end{lemma}

\begin{proof}
We provide a short proof of \eqref{eq:slfg_kernel_dim}. Let  
\begin{align}
\psi = \begin{pmatrix}
\psi_1 \\
\psi_2
\end{pmatrix} \in H_1 \oplus H_2. \nonumber
\end{align}
Suppose $\psi \in \operatorname{ker} A$. Then, $A \psi = 0$ is equivalent to
\begin{align}
A_{11} \psi_1 + A_{12} \psi_2 &= 0, \nonumber \\
A_{21} \psi_1 + A_{22} \psi_2 &= 0. \nonumber
\end{align} 
Since $A_{22}^{-1}$ is invertible, this implies $\psi_2 = - A_{22}^{-1} A_{21} \psi_1$ and thus $S \psi_1 = A_{11} \psi_1 - A_{12} A_{22}^{-1} A_{21} \psi_1 = 0$. 

Conversely, if $\phi \in \operatorname{ker} S$, then 
\begin{align}
\psi = \begin{pmatrix}
\phi \\
-A_{22}^{-1} A_{21} \phi 
\end{pmatrix} \in \operatorname{ker} A. \nonumber
\end{align} 

\end{proof}

If $A^{-1}$ is bounded, then
\begin{align}
A^{-1} = \begin{pmatrix}
S^{-1} &  -S^{-1} A_{12} A_{22}^{-1} \\
- A_{22}^{-1} A_{21}S^{-1} & A_{22}^{-1 } - A_{22}^{-1} A_{21} S^{-1} A_{12} A_{22}^{-1}
\end{pmatrix}. \label{eq:slfg_formula_inversese}
\end{align}

While we will not use equation \eqref{eq:slfg_formula_inversese}, we will make extensive use of the invertibility condition that $A$ is invertible if and only if its Schur complement $S$ is invertible (having shown that $A_{22}$ is invertible) and that the dimensions of the kernels of $A$ and $S$ coincide.

\section{Proofs of the main results} \label{sec:weak_coupling_proof}

We will now give the proofs of Theorems \ref{thm:expansion_zero}, \ref{thm:expansion_nonint} and \ref{thm:expansion_int}. First, to shorten notation, we define $v:=V^\frac{1}{2}$, which is well-defined, since $V\geq 0$.

We begin by applying the Birman-Schwinger principle, see the recent monograph \cite{Frank2022} and references therein. It reveals that $\lambda < 0$ is an eigenvalue of $H(\varepsilon)$ if and only if $1$ is an eigenvalue of the Birman-Schwinger operator $\varepsilon v(H(\varepsilon)-\lambda)^{-1}v$. Since 
\begin{align}
\varepsilon v(H(\varepsilon)-\lambda)^{-1}v = \begin{pmatrix}
\varepsilon v(P_+(A)-\lambda)^{-1}v & 0 \\
0 & \varepsilon v(P_-(A)-\lambda)^{-1}v
\end{pmatrix}, \nonumber
\end{align}
this is the case if and only if $1$ is an eigenvalue of $\varepsilon v(P_+(A)-\lambda)^{-1}v$ or $1$ is an eigenvalue of $\varepsilon v(P_-(A)-\lambda)^{-1}v$. Here, the linear operators
\begin{align}
 L^2(\R^2) \ni \psi &\mapsto v(P_\pm(A)-\lambda)^{-1} v \psi \in L^{2}(\R^2) , \qquad \lambda <0,  \nonumber
\end{align}
are bounded, since $V$ is bounded by Assumption \ref{assump:potentialV}. They are in fact compact. This follows from compactness of the Birman-Schwinger operator $v(-\Delta-\lambda)^{-1}v$ to the corresponding classical Schrödinger operator (which is actually Hilbert-Schmidt, since $V$ satisfies the conditions of Proposition 3.2 of \cite{Simon1976} by Assumption \ref{assump:potentialV}). Together with the diamagnetic inequality, see \cite{Hundertmark2004}, and a theorem by Pitt \cite{Pitt1979} one concludes that the Birman-Schwinger operator $v((-i\nabla +A)^2-\lambda)^{-1}v$ to the corresponding magnetic Schrödinger operator is also compact. The Birman-Schwinger operators $v(P_\pm(A)-\lambda)^{-1}v=v((-i\nabla +A)^2\pm B-\lambda)^{-1}v$ for the spin-components of the Pauli operator are then seen to be compact as well when one applies a resolvent identity (note that $B$ is bounded due to Assumption \ref{assump:magnetic_field}).

Now, recall that the asymptotic expansions of the Theorems \ref{thm:resolvent_int_zero} to  \ref{thm:resolvent_int} hold in $\mathcal{B}(s,-s)$, $s>3$. Under Assumption \ref{assump:potentialV} on the potential $V$, the linear operators
\begin{align} \label{eq:operators_with_v_multiplied}
L^2(\R^2) \ni \psi &\mapsto v \psi \in L^{2,\sigma}(\R^2), \\
L^{2,-\sigma}(\R^2) \ni \psi &\mapsto v \psi \in L^{2}(\R^2)
\end{align}
are bounded. If $v(P_\pm(A)-\lambda)^{-1} v$ is understood as a map 
\begin{align}
L^2(\R^2) \overset{v}{ \longrightarrow} L^{2,\sigma} (\R^2)  \overset{(P_\pm(A)-\lambda)^{-1}}{ \longrightarrow} L^{2,-\sigma} (\R^2)  \overset{v}{ \longrightarrow} L^{2} (\R^2),  \nonumber
\end{align} 
one sees that it is valid to take the asymptotic expansions of Theorem \ref{thm:resolvent_int_zero} to  \ref{thm:resolvent_int} in $\mathcal{B}(\sigma,-\sigma)$ and simply multiply them by $v$ from the left and right to gain asymptotic expansions of $v(P_\pm(A)-\lambda)^{-1} v \in \mathcal{L}(L^2(\R^2))$ as $\lambda\to 0$.

\subsection{Zero flux - Proof of Theorem \ref{thm:expansion_zero}}

Let us start with the case $\alpha = 0$. We treat $P_+(A)$ and $P_-(A)$ simultaneously and divide the proof into several steps. \\

\textbf{Step 1: Preliminaries}\\

 Note that if $\alpha =0$, then $\psi_0^\pm = e^{\mp h+i\chi} = O(1)$ as $|x| \to \infty$. Hence, under Assumption \ref{assump:potentialV} on the potential $V$, we have $v\psi_0^\pm \in L^2(\R^2)$. Consider $H_\pm = \operatorname{span}\{ \varphi^\pm \}\subset L^2(\R^2)$ where $\varphi^\pm = v\psi_0^\pm / \Vert v\psi_0^\pm \Vert_{L^2(\R^2)}$. Note that $v(P_{\pm}(A)-\lambda)^{-1}v \in \mathcal{L}(L^2(\R^2))$ is continuous in operator norm  with respect to $\lambda < 0$ and 
\begin{align} \label{eq:vP-lv_l-inf}
v(P_{\pm}(A)-\lambda)^{-1}v = o(1)
\end{align}
in $\mathcal{L}(L^2(\R^2))$ as $\lambda	\to -\infty$. Furthermore, by Theorem \ref{thm:resolvent_int_zero},
\begin{align} \label{eq:vP-lv_l-0}
v(P_{\pm}(A)-\lambda)^{-1}v= - \frac{\log |\lambda|}{4 \pi}   v\psi_0^\pm \langle v\psi_0^\pm, \, . \,  \rangle   +O(1)
\end{align}
in $\mathcal{L}(L^2(\R^2))$ as $\lambda \nearrow 0$. This means that
\begin{align}
&v(P_{\pm}(A)-\lambda)^{-1}v|_{H_\pm \to H_\pm^\perp}, \label{eq:vpv_12} \\
&v(P_{\pm}(A)-\lambda)^{-1}v|_{H_\pm^\perp \to H_\pm}, \label{eq:vpv_21} \\
&v(P_{\pm}(A)-\lambda)^{-1}v|_{H_\pm^\perp \to H_\pm^\perp} \label{eq:vpv_22}
\end{align}
are uniformly bounded over $\lambda < 0$. \\

\textbf{Step 2: Bounding the number of negative eigenvalues of $H(\varepsilon)$}\\

We can already show with \eqref{eq:vP-lv_l-0} that $H(\varepsilon)$ has at most two eigenvalues for small enough $\varepsilon$. For that, we argue that both spin components $P_\pm(A) - \varepsilon V$ have at most one negative eigenvalue. Let $\lambda_k^\pm(\varepsilon)$, $k=0,1,...$, denote the negative eigenvalues of $P_\pm(A) - \varepsilon V$ and $\mu_k^\pm(\lambda)$, $k=0,1,...$, the non-negative eigenvalues of $v(P_{\pm}(A)-\lambda)^{-1}v$. Let also $\tilde{\mu}_k^\pm(\lambda)$, $k=0,1,...$, denote the eigenvalues of the leading order operator 
\begin{align} 
S(\lambda) =- \frac{\log |\lambda|}{4 \pi}   v\psi_0^\pm \langle v\psi_0^\pm, \, . \,  \rangle   \nonumber
\end{align}
on the right hand side of \eqref{eq:vP-lv_l-0}. Note that all eigenvalues of $S(\lambda)$ except one are zero since $S(\lambda)$ is a rank one operator. We can assume that $|\lambda|$ is small enough such that $\tilde{\mu}_k^\pm(\lambda)\geq 0$ for any $k$. We also assume that the eigenvalues $\mu_k^\pm(\lambda)$, $\tilde{\mu}_k^\pm(\lambda)$ are indexed in decreasing order. Then $\tilde{\mu}_k^\pm(\lambda)=0$ for $k\neq 0$ and $\tilde{\mu}_0^\pm(\lambda)$ is the only eigenvalue that is possibly non-zero. The asymptotic equation \eqref{eq:vP-lv_l-0} implies that there exist $\Lambda<0$ and $C>0$ such that for any $\Lambda < \lambda < 0$
\begin{align}
|\mu_k^\pm(\lambda) - \tilde{\mu}_k^\pm(\lambda) | \leq C. \nonumber
\end{align}
and therefore  $|\mu_k^\pm(\lambda) | \leq C$ for any $\Lambda < \lambda < 0$ and $k\neq 0$. It follows that $
|\varepsilon \mu_k^\pm(\lambda) | \leq 1 $ for $k\neq 0$ and $|\lambda|, \varepsilon $ small enough. By the Birman-Schwinger principle, we conclude
\begin{align}
\#\{ \lambda_k^\pm(\varepsilon)  \} &= \lim_{\lambda \nearrow 0} \# \{ \lambda_k^\pm(\varepsilon) \, : \, \lambda_k^\pm(\varepsilon)< \lambda  \} = \lim_{\lambda \nearrow 0} \# \{ \varepsilon\mu_k^\pm(\lambda) \, : \, \varepsilon \mu_k^\pm(\lambda)>1 \} \nonumber \\
&\leq \lim_{\lambda \nearrow 0} \# \{ \varepsilon\mu_0^\pm(\lambda) \} = 1. \nonumber
\end{align}
for $\varepsilon $ small enough. This shows that $P_\pm(A) - \varepsilon V$ each have at most one negative eigenvalue. \\

Now we show existence of the negative eigenvalues of $H(\varepsilon)$ and derive their asymptotic expansions. For that, we consider the operator 
\begin{align}
K_\varepsilon^\pm(\lambda) := 1 - \varepsilon v(P_\pm(A)-\lambda)^{-1}v. \label{eq:K_eps_definition}
\end{align}
Because the Birman-Schwinger operator $\varepsilon v(P_\pm(A)-\lambda)^{-1}v$ is compact and self-adjoint, the condition of $1$ being an eigenvalue of $\varepsilon v(P_\pm(A)-\lambda)^{-1}v$ is equivalent to the operator $K_\varepsilon^\pm(\lambda)$ being not invertible on $L^2(\R^2)$. The invertibility of $K_\varepsilon^\pm(\lambda)$ will therefore be what we examine in the following. \\

\textbf{Step 3: Reduction to a rank-one operator}\\

We decompose $K_\varepsilon^\pm(\lambda)$ into
\begin{align}
K_{\varepsilon}^\pm(\lambda)= \begin{pmatrix}
K_{\varepsilon,11}^\pm(\lambda) & K_{\varepsilon,12}^\pm(\lambda)\\[5pt]
K_{\varepsilon,21}^\pm(\lambda) & K_{\varepsilon,22}^\pm(\lambda)
\end{pmatrix} = \begin{pmatrix}
K_\varepsilon^\pm(\lambda)|_{H_\pm \to H_\pm} & K_\varepsilon^\pm(\lambda)|_{H_\pm^\perp \to H_\pm} \\
K_\varepsilon^\pm(\lambda)|_{H_\pm \to H_\pm^\perp} & K_\varepsilon^\pm(\lambda)|_{H_\pm^\perp \to H_\pm^\perp}
\end{pmatrix}. \nonumber
\end{align}

Using the uniform boundedness of \eqref{eq:vpv_12}, \eqref{eq:vpv_21} and \eqref{eq:vpv_22} and $1|_{H_\pm^\perp \to H_\pm} = 1|_{H_\pm \to H_\pm^\perp} = 0$, we get the estimates
\begin{align} 
\Vert K_{\varepsilon,12}^\pm(\lambda) \Vert &= \Vert (  - \varepsilon v(P_{\pm}(A)-\lambda)^{-1}v )|_{H_\pm^\perp \to H_\pm} \Vert \lesssim \varepsilon, \label{eq:K_eps12}\\
\Vert K_{\varepsilon,21}^\pm(\lambda) \Vert&=\Vert (  - \varepsilon v(P_{\pm}(A)-\lambda)^{-1}v )|_{H_\pm \to H_\pm^\perp}\Vert \lesssim \varepsilon , \label{eq:K_eps21}\\
\Vert K_{\varepsilon,22}^\pm(\lambda) -1|_{H_\pm^\perp \to H_\pm^\perp} \Vert&= \Vert ( - \varepsilon v(P_{\pm}(A)-\lambda)^{-1}v )|_{H_\pm^\perp \to H_\pm^\perp} \Vert \lesssim \varepsilon  \label{eq:K_eps22}
\end{align}
for any $\lambda <0$. Thus, if $\varepsilon$ is sufficiently small, $K_{\varepsilon,22}^\pm(\lambda)$ becomes invertible on $H_\pm^\perp$ for any $\lambda <0$ and then
\begin{align}\label{eq:K_eps22_inv}
\Vert (K_{\varepsilon,22}^\pm(\lambda) )^{-1} - 1 |_{H_\pm^\perp \to H_\pm^\perp} \Vert \lesssim \varepsilon .
\end{align}
The SLFG Lemma (Lemma \ref{lem:slfg_formula}) now implies that $K_{\varepsilon}^\pm(\lambda)$ is invertible if and only if its Schur complement
\begin{align}\label{eq:S_eps}
S_{\varepsilon}^\pm(\lambda) := K_{\varepsilon,11}^\pm(\lambda) - K_{\varepsilon,12}^\pm(\lambda) (K_{\varepsilon,22}^\pm(\lambda))^{-1}K_{\varepsilon,21}^\pm(\lambda)
\end{align}
is invertible in $H_\pm$. Here, 
\begin{align} \label{eq:K_eps122221_comb}
\Vert K_{\varepsilon,12}^\pm(\lambda) (K_{\varepsilon,22}^\pm(\lambda))^{-1}K_{\varepsilon,21}^\pm(\lambda) \Vert \lesssim \varepsilon^2
\end{align}
by \eqref{eq:K_eps12}, \eqref{eq:K_eps21} and \eqref{eq:K_eps22_inv}. Observe that $S_{\varepsilon}^\pm(\lambda): H^\pm \to H^\pm$ is a self-adjoint rank one operator, so 
\begin{align}
S_{\varepsilon}^\pm(\lambda) = s_\varepsilon^\pm(\lambda) \varphi^\pm \langle \varphi^\pm, \, . \, \rangle \nonumber
\end{align}
where 
\begin{align} \label{eq:s_eps_pm_small}
s_\varepsilon^\pm(\lambda) = \langle \varphi^\pm,S_{\varepsilon}^\pm(\lambda) \varphi^\pm \rangle  \in \R.
\end{align}
Hence, the Schur complement $S_{\varepsilon}^\pm(\lambda) $ is not invertible in $H^\pm$ if and only if 
\begin{align} \label{eq:flux0_scalar_s_eps}
s_\varepsilon^\pm(\lambda) = 0.
\end{align}
The negative eigenvalues of $H(\varepsilon)$ are given by the solutions $\lambda$ to the scalar equation \eqref{eq:flux0_scalar_s_eps}.\\

\textbf{Step 4: Showing existence of negative eigenvalues of $H(\varepsilon)$}\\

Let us argue that $s_\varepsilon^\pm(\lambda) = 0$ has at least one solution $\lambda<0$ when $\varepsilon$ is small enough (one each for the plus and the minus case). Let $\lambda^*<0$ be given. Because of \eqref{eq:vP-lv_l-inf}, $v(P_{\pm}(A)-\lambda)^{-1}v$ is uniformly bounded in $\mathcal{L}(L^2(\R^2))$ for $\lambda < \lambda^*<0$. With this fact and equations \eqref{eq:S_eps}, \eqref{eq:K_eps122221_comb} and \eqref{eq:s_eps_pm_small}, we conclude that
\begin{align}
s_\varepsilon^\pm(\lambda) \geq 1 - c_1 \varepsilon -c_2 \varepsilon^2, \qquad \lambda < \lambda^*, \nonumber
\end{align}
for some constants $c_1,c_2>0$ and in particular
\begin{align}
s_\varepsilon^\pm(\lambda) > 0, \qquad \lambda < \lambda^*, \nonumber
\end{align}
if $\varepsilon$ is small enough. On the other hand, because of \eqref{eq:vP-lv_l-0}, \eqref{eq:S_eps} and \eqref{eq:K_eps122221_comb}, there exists some $\Lambda^* <0$ such that
\begin{align} \label{eq:s_eps_pm_lambda_small}
s_\varepsilon^\pm(\lambda) &= 1 + \frac{\varepsilon}{4\pi} \Vert v\psi_0^\pm \Vert_{L^2(\R^2)}^2 \log|\lambda| + r_\varepsilon(\lambda), \qquad \Lambda^* < \lambda < 0,
\end{align}
with a remainder term $r_\varepsilon(\lambda) $ that satisfies $|r_\varepsilon(\lambda)| \lesssim \varepsilon $ for $ \Lambda^* < \lambda < 0$. Therefore, $s_\varepsilon^\pm(\lambda)  \to - \infty$ as $\lambda \nearrow 0$. Since $s_\varepsilon^\pm(\lambda)$ is continuous in $\lambda<0$, the intermediate value theorem implies that there exists at least one solution $\lambda_0^\pm(\varepsilon)\in [\lambda^*,0)$ of $s_\varepsilon^\pm(\lambda) = 0$. 

Since we have shown in Step 2 that $P_\pm(A)-\varepsilon V$ each have at most one eigenvalue for small enough $\varepsilon$, we have proven that $P_\pm(A)-\varepsilon V$ each have precisely one negative eigenvalue for small enough $\varepsilon$ and hence $H(\varepsilon)$ has precisely two negative eigenvalues for small enough $\varepsilon$. \\

\textbf{Step 5: Extracting eigenvalue asymptotics}\\

Note that $\lambda_0^\pm(\varepsilon) \to 0$ as $\varepsilon \searrow 0$, since $\lambda^*$ was arbitrary. It follows that $\Lambda^*<\lambda_0^\pm(\varepsilon)<0$ for any $\varepsilon$ small enough and according to \eqref{eq:s_eps_pm_lambda_small}, the solution $\lambda_0^\pm(\varepsilon)$ must satisfy
\begin{align}
\log|\lambda_0^\pm(\varepsilon)| = - \frac{4\pi}{\Vert v\psi_0^\pm \Vert_{L^2(\R^2)}^2} \varepsilon^{-1} (1+r_\varepsilon(\lambda_0^\pm(\varepsilon))). \nonumber
\end{align}
Because $|r_\varepsilon(\lambda)| \lesssim \varepsilon $ for $\Lambda^*<\lambda_0^\pm(\varepsilon)<0$, we finally obtain 
\begin{align}
\lambda_0^\pm(\varepsilon) = -\exp\left( -  \mu_{\pm} ^{-1} \, \varepsilon^{-1} (1+O(\varepsilon)) \right), \qquad \varepsilon \searrow 0, \nonumber
\end{align}
with
\begin{align}
\mu_{\pm} = \frac{1}{4\pi}  \Vert v\psi_0^\pm \Vert_{L^2(\R^2)}^2 = \frac{1}{4\pi}\int_{\R^2} V|\psi_0^\pm |^2 \dd x.\nonumber
\end{align}
\begin{flushright}
$\square$
\end{flushright}

\subsection{Non-integer flux - Proof of Theorem \ref{thm:expansion_nonint}}

For non-zero flux, the calculations become more complicated. We first recall that the asymptotic expansion of $P_+(A)$ has no singularities at $\lambda = 0$, while that of $P_-(A)$ has, so we need to treat the spin-up component and the spin-down component separately.

Let us first consider the component $P_+(A)$ of the Pauli operator. We consider $K_\varepsilon^+(\lambda)$ as defined by \eqref{eq:K_eps_definition}. Theorem \ref{thm:resolvent_plus} gives for any $\alpha>0$
\begin{align}
v (P_+(A)-\lambda)^{-1} v = O(1) \nonumber
\end{align}
in $\mathcal{L}(L^2(\R^2))$ as $\lambda \nearrow 0$. Since also $v (P_+(A)-\lambda)^{-1} v = o(1)$ as $\lambda\to -\infty$ and $v (P_+(A)-\lambda)^{-1} v$ is continuous with respect to $\lambda<0$, we find that $v (P_+(A)-\lambda)^{-1} v$ is uniformly bounded for $\lambda <0$. This implies that $K_\varepsilon^+(\lambda)=1 + O(\varepsilon)$ in $\mathcal{L}(L^2(\R^2))$ as $\varepsilon \searrow 0$ uniform in $\lambda <0$ and hence $K_\varepsilon^+(\lambda)$ is invertible for any $\lambda <0$ as soon as $\varepsilon$ is small enough. This means that $P_+(A) - \varepsilon V$ exhibits no eigenvalue $\lambda<0$ if $\varepsilon$ is small enough. 

For small enough $\varepsilon$, any negative eigenvalues of $H(\varepsilon)$ must therefore be negative eigenvalues of $P_-(A) - \varepsilon V$. Let us discuss $P_-(A) - \varepsilon V$ and its associated operator $K_\varepsilon^-(\lambda)$ in the following.\\

\textbf{Step 1: Preliminaries }\\

Let $\alpha >0$, $\alpha \in \R \setminus \mathbb{Z}$. Let $n=\lfloor \alpha \rfloor$ and $\alpha' = \alpha - n$. Finally, recall the (generalized) Aharonov-Casher states $\psi_k^-(x) = (x_1-ix_2)^{k-1} e^{h(x)+i\chi(x)}$ for $k=0,...,n$. The states $\{ \psi_k^- \}_{k=0}^{n-1}$ span the zero eigenspace of $P_-(A)$, i.e.\ $\operatorname{ran } P_0^-$, while $\psi_{n}^-$ is a virtual bound state at zero. Recall that the (generalized) Aharonov-Casher states are linearly independent under Assumption \ref{assump:magnetic_field} on the magnetic field. Under Assumption \ref{assump:potentialV} on the potential $V$, we have $v \psi_k^- \in L^2(\R^2)$, $k=0,...,n$, since all $\{\psi_k^-\}_{k=0}^n$ are bounded functions. Furthermore, the functions $\{ v\psi_k^-\}_{k=0}^n$ are linearly independent, since $V>0$ on some set of non-zero measure.

Let $H_{n} = \operatorname{span}\{ v\psi_k^- \}_{k=0}^{n-1} $ and $H_{n+1} = \operatorname{span}\{ v\psi_k^- \}_{k=0}^{n} $. Because of linear independence of the states $v\psi_k^-$, the spaces $H_{n}$ and $H_{n+1}$ are $n$- resp.~$(n+1)$-dimensional subspaces of $L^2(\R^2)$. To shorten notation, for $\lambda<0$, $|\lambda|$ small enough, let also
\begin{align}
f_0(\lambda)&= -\lambda^{-1}, \nonumber \\
f_1(\lambda)&=\frac{\omega(1+\alpha')\lambda^{\alpha'-1}}{1+\tau \omega(1+\alpha') \lambda^{\alpha'}}, \nonumber \\
f_2(\lambda)&= - \frac{\zeta(\alpha') \lambda^{-\alpha'}}{1+\rho \zeta(\alpha') \lambda^{1-\alpha'}}  \nonumber 
\end{align}
with $\zeta(t)$ and $\omega(t)$ given by \eqref{eq:zeta_omega_definition}.
These are the prefactors appearing in front the of the linear operators in the resolvent expansion of Theorem \ref{thm:resolvent_non_int}. Observe that for $\lambda<0$ the prefactors $f_0(\lambda)$, $f_1(\lambda)$ and $f_2(\lambda)$ can be rephrased as
\begin{align}
f_0(\lambda) &= |\lambda|^{-1}, \nonumber \\
f_1(\lambda) &=  \frac{|d_n^-|^2e^{-i\pi(1-\alpha')}|\lambda|^{-(1-\alpha')}}{\zeta(1+\alpha')+\tau |d_n^-|^2 e^{i\pi \alpha'}|\lambda|^{\alpha'}} =  -\frac{|d_n^-|^2|\lambda|^{-(1-\alpha')}}{\frac{4^{\alpha'} \Gamma(1+\alpha')}{\pi \Gamma(-\alpha')}+\tau |d_n^-|^2 |\lambda|^{\alpha'}}, \nonumber \\
f_2(\lambda)&= - \frac{\zeta(\alpha') e^{-i\pi \alpha'}|\lambda|^{-\alpha'}}{1+\rho \zeta(\alpha') e^{i\pi (1-\alpha')}|\lambda|^{1-\alpha'}}=  \frac{ \frac{4^{\alpha'-1} \Gamma(\alpha') }{\pi \Gamma(1-\alpha')} |\lambda|^{-\alpha'}}{1+\rho \frac{4^{\alpha'-1} \Gamma(\alpha') }{\pi \Gamma(1-\alpha')} |\lambda|^{1-\alpha'}}, \nonumber
\end{align}
and one sees that all three prefactors are real-valued.

By Theorem \ref{thm:resolvent_non_int} and by boundedness of the operators \eqref{eq:operators_with_v_multiplied}, we find
\begin{align} \label{eq:vPv_non_int_expansion}
v(P_-(A)-\lambda)^{-1}v&= f_0(\lambda) vP_0^-v +f_1(\lambda) v\psi^- \langle v\psi^-, . \rangle + f_2(\lambda) v\varphi^- \langle v\varphi^-, . \rangle  + O(1)
\end{align}
in $\mathcal{L}(L^2(\R^2))$ as $\lambda \nearrow 0$, where $v\psi^- \in H_n$ and $v\varphi^- \in H_{n+1} \setminus H_{n}$. Since $f_0(\lambda)$, $f_1(\lambda)$, $f_2(\lambda)$  are real for $\lambda<0$, all terms on the right-hand side are self-adjoint. Also, note that $\operatorname{ran } vP_0^-v = H_n$ and therefore since $vP_0^-v$ is self-adjoint, $\operatorname{ker }vP_0^-v = H_n^\perp$. In addition, $H_n^\perp \subset \operatorname{ker }v\psi^- \langle v\psi^-, . \rangle $, since $v\psi^- \in H_n$. 

As in the case of zero flux, we note that $v(P_{\pm}(A)-\lambda)^{-1}v$ is continuous in operator norm  with respect to $\lambda < 0$ and 
\begin{align} 
v(P_{-}(A)-\lambda)^{-1}v = o(1) \nonumber
\end{align}
in $\mathcal{L}(L^2(\R^2))$ as $\lambda	\to -\infty$. Hence, we conclude that
\begin{align}
&v(P_{-}(A)-\lambda)^{-1}v|_{H_{n+1} \to H_{n+1}^\perp}, \nonumber\\
&v(P_{-}(A)-\lambda)^{-1}v|_{H_{n+1}^\perp \to H_{n+1}}, \nonumber \\
&v(P_{-}(A)-\lambda)^{-1}v|_{H_{n+1}^\perp \to H_{n+1}^\perp} \nonumber
\end{align}
are uniformly bounded over $\lambda < 0$. \\

\textbf{Step 2: Bounding the number of negative eigenvalues of $H(\varepsilon)$}\\

Similarly as in the case $\alpha=0$, we can already show that $H(\varepsilon)$ has no more than $n+1$ negative eigenvalues. We have already seen that $P_+(A)-\varepsilon V$ contributes no negative eigenvalues for small enough $\varepsilon$. We now argue that $P_-(A)-\varepsilon V$ has at most $n+1$ negative eigenvalues for small enough $\varepsilon$. Hence, let as before $\lambda_k^-(\varepsilon)$, $k=0,1,...$, denote the negative eigenvalues of $P_-(A) - \varepsilon V$ and $\mu_k^-(\lambda)$, $k=0,1,...$, the non-negative eigenvalues of $v(P_{-}(A)-\lambda)^{-1}v$. Let $\tilde{\mu}_k^-(\lambda)$, $k=0,1,...$, now denote the eigenvalues of the operator 
\begin{align} \label{eq:singular_part_nonint}
S(\lambda)=f_0(\lambda) vP_0^-v +f_1(\lambda) v\psi^- \langle v\psi^-, . \rangle + f_2(\lambda) v\varphi^- \langle v\varphi^-, . \rangle.
\end{align}
It collects all terms on the right hand side of \eqref{eq:vPv_non_int_expansion} that are singular as $\lambda \to 0$. Since $S(\lambda)$ is at most rank $n+1$, it has at most $n+1$ non-zero eigenvalues. We can again assume that $|\lambda|$ is small enough such that $\tilde{\mu}_k^-(\lambda)\geq 0$ for any $k$. This is because  $f_0(\lambda)$, $f_2(\lambda)$ are non-negative and $f_1(\lambda) v\psi^- \langle v\psi^-, . \rangle$ can be seen as a weak perturbation to $f_0(\lambda) vP_0^-v$ since $v\psi^-\in H_n$ and $f_1(\lambda)$ has a weaker singularity as $\lambda \to 0$ than $f_0(\lambda)$. If we assume that the eigenvalues $\mu_k^-(\lambda)$, $\tilde{\mu}_k^-(\lambda)$ are indexed in decreasing order, then $\tilde{\mu}_k^-(\lambda)=0$ for $k > n$. Equation \eqref{eq:vPv_non_int_expansion} implies that there exist $\Lambda>0$ and $C>0$ such that for any $\Lambda < \lambda < 0$
\begin{align}
|\mu_k^-(\lambda) - \tilde{\mu}_k^-(\lambda) | \leq C, \nonumber
\end{align}
and therefore  $|\mu_k^-(\lambda) | \leq C$ for any $\Lambda < \lambda < 0$ and $k> n$. It follows that $
|\varepsilon \mu_k^-(\lambda) | \leq 1 $ for $k> n$ and $|\lambda|, \varepsilon $ small enough. We finally conclude for $\varepsilon $ small enough
\begin{align}
\#\{ \lambda_k^\pm(\varepsilon)  \} &= \lim_{\lambda \nearrow 0} \# \{ \lambda_k(\varepsilon) \, : \, \lambda_k(\varepsilon)< \lambda  \} = \lim_{\lambda \nearrow 0} \# \{ \varepsilon\mu_k^\pm(\lambda) \, : \, \varepsilon \mu_k^\pm(\lambda)>1 \} \nonumber \\
& \leq  \lim_{\lambda \nearrow 0}  \# \{ \varepsilon\mu_0^\pm(\lambda),...,\varepsilon\mu_n^\pm(\lambda) \} = n+1 \nonumber
\end{align}
by the Birman-Schwinger principle. This shows that $P_-(A) - \varepsilon V$ and hence $H(\varepsilon)$ has at most $n+1$ negative eigenvalues. \\

\textbf{Step 3: Reduction to a finite rank operator}\\

We now consider the operator 
\begin{align}
K_\varepsilon^-(\lambda) = 1 - \varepsilon v(P_-(A)-\lambda)^{-1}v. \nonumber
\end{align}
As argued in the case $\alpha = 0$, $\lambda$ is an eigenvalue of $H(\varepsilon)$ if and only if the operator $K_\varepsilon^-(\lambda)$ is not invertible. 
We decompose $K_\varepsilon^-(\lambda)$ into
\begin{align}
K_{\varepsilon}^-(\lambda)= \begin{pmatrix}
K_{\varepsilon,11}^-(\lambda) & K_{\varepsilon,12}^-(\lambda)\\[5pt]
K_{\varepsilon,21}^-(\lambda) & K_{\varepsilon,22}^-(\lambda)
\end{pmatrix} = \begin{pmatrix}
K_\varepsilon^-(\lambda)|_{H_{n+1} \to H_{n+1}} & K_\varepsilon^-(\lambda)|_{H_{n+1}^\perp \to H_{n+1}} \\
K_\varepsilon^-(\lambda)|_{H_{n+1} \to H_{n+1}^\perp} & K_\varepsilon^-(\lambda)|_{H_{n+1}^\perp \to H_{n+1}^\perp}
\end{pmatrix}. \nonumber
\end{align}
Using $1|_{H_{n+1}^\perp \to H_{n+1}} = 1|_{H_{n+1} \to H_{n+1}^\perp} = 0$, we find
\begin{align} 
\Vert K_{\varepsilon,12}^-(\lambda) \Vert &= \Vert (  - \varepsilon v(P_{-}(A)-\lambda)^{-1}v )|_{H_{n+1}^\perp \to H_{n+1}} \Vert \lesssim \varepsilon, \label{eq:K_eps12_nonint}\\
\Vert K_{\varepsilon,21}^-(\lambda) \Vert&=\Vert (  - \varepsilon v(P_{-}(A)-\lambda)^{-1}v )|_{H_{n+1} \to H_{n+1}^\perp}\Vert \lesssim \varepsilon, \label{eq:K_eps21_nonint}\\
\Vert K_{\varepsilon,22}^-(\lambda) -1|_{H_{n+1}^\perp \to H_{n+1}^\perp} \Vert&= \Vert ( - \varepsilon v(P_{-}(A)-\lambda)^{-1}v )|_{H_{n+1}^\perp \to H_{n+1}^\perp} \Vert \lesssim \varepsilon \label{eq:K_eps22_nonint}
\end{align}
for any $\lambda <0$. If $\varepsilon$ is sufficiently small, then $K_{\varepsilon,22}^-(\lambda)$ is invertible for any $\lambda <0$ and then
\begin{align}\label{eq:K_eps22_inv_nonint}
\Vert (K_{\varepsilon,22}^-(\lambda) )^{-1} - 1 |_{H_{n+1}^\perp \to H_{n+1}^\perp} \Vert \lesssim \varepsilon.
\end{align}
The SLFG formula (Lemma \ref{lem:slfg_formula}) now implies that $K_{\varepsilon}^-(\lambda)$ is invertible if and only its Schur complement
\begin{align}\label{eq:S_eps_nonint}
S_{\varepsilon}^-(\lambda) = K_{\varepsilon,11}^-(\lambda) - K_{\varepsilon,12}^-(\lambda) (K_{\varepsilon,22}^-(\lambda))^{-1}K_{\varepsilon,21}^-(\lambda)
\end{align}
is invertible in $H_{n+1}$. Here, 
\begin{align} \label{eq:K_eps122221_comb_nonint}
\Vert K_{\varepsilon,12}^-(\lambda) (K_{\varepsilon,22}^-(\lambda))^{-1}K_{\varepsilon,21}^-(\lambda) \Vert \lesssim \varepsilon^2
\end{align}
by \eqref{eq:K_eps12_nonint}, \eqref{eq:K_eps21_nonint} and \eqref{eq:K_eps22_inv_nonint}. The Schur complement $S_{\varepsilon}^-(\lambda)$ is finite rank and self-adjoint and therefore not invertible if and only if one of its eigenvalues is zero. Let us denote the eigenvalues of the Schur complement by $\{ \mu_k (S_{\varepsilon}^-(\lambda)) \}_{k=0}^n$. Any solution $\lambda$ to one of the $n+1$ equations 
\begin{align} \label{eq:eigeval_schurcomplt_S_k}
\mu_k (S_{\varepsilon}^-(\lambda))  = 0, \qquad k=0,...,n,
\end{align}
yields therefore a negative eigenvalue of $H(\varepsilon)$.\\

\textbf{Step 4: Showing existence of negative eigenvalues of $H(\varepsilon)$}\\

Each parametric eigenvalue function $ \lambda \mapsto \mu_k (S_{\varepsilon}^-(\lambda)) $, $k=0,..., n$, is continuous in $\lambda< 0$ and it is not difficult to show by an intermediate value theorem argument similar to that in the case $\alpha=0$, that each equation  $\mu_k (S_{\varepsilon}^-(\lambda)) =0 $ must have at least one solution $\lambda_k(\varepsilon)$, if $\varepsilon$ is sufficiently small. Since we have already shown that $H(\varepsilon)$ cannot have more than $n+1$ negative eigenvalues $\lambda_k(\varepsilon)$, it is clear that if the solutions $\lambda_k(\varepsilon)$ for each $k=0,...,n$ are distinct, then $H(\varepsilon)$ must have precisely $n+1$ negative eigenvalues, namely said solutions $\lambda_k(\varepsilon)$, $k=0,..., n$. However, it may happen that the solutions $\lambda_k(\varepsilon)$ of \eqref{eq:eigeval_schurcomplt_S_k} are not distinct. In this case, suppose $m$ equations $\mu_k (S_{\varepsilon}^-(\lambda))  = 0$ have the same solution, i.e.\ $\lambda_\ell(\varepsilon)=...=\lambda_{\ell+m-1}(\varepsilon)$ for some $\ell$ and hence $S_{\varepsilon}^-(\lambda)|_{\lambda=\lambda_\ell(\varepsilon)}$ has a zero eigenvalue of multiplicity $m$. Using \eqref{eq:slfg_kernel_dim} from the SLFG Lemma, we are allowed to conclude that if $S_{\varepsilon}^-(\lambda)|_{\lambda=\lambda_\ell(\varepsilon)}$ has a zero eigenvalue of multiplicity $m$, then also  $K_{\varepsilon}^-(\lambda)|_{\lambda=\lambda_\ell(\varepsilon)}$ has a zero eigenvalue of multiplicity $m$, which in turn means that $P_-(A)-\varepsilon V$ has the eigenvalue $\lambda=\lambda_\ell(\varepsilon)$ of multiplicity $m$ by the Birman-Schwinger principle. Alternatively, we can say that $P_-(A)-\varepsilon V$ has the eigenvalues $\lambda_\ell(\varepsilon)=...=\lambda_{\ell+m-1}(\varepsilon)$ counted with multiplicity. The number of negative eigenvalues of $H(\varepsilon)$ counted with multiplicities is therefore always precisely $n+1$. Each of them comes as one solution of one of the equations \eqref{eq:eigeval_schurcomplt_S_k}.\\

Although we have now argued the existence of precisely $n+1$ negative eigenvalues of $H(\varepsilon)$, we still have difficulties to directly extract the asymptotic behaviour, since the asymptotic expansion of the resolvent from \eqref{eq:vPv_non_int_expansion} projected onto $H_{n+1}$ has still multiple singular terms of varying degree.

To get hold of asymptotics, we consider different ``reference windows'' for $\lambda$ that scale with $\varepsilon$ so that $\varepsilon f_i(\lambda)$ is bounded above and below for a particular $i$. This allows us to seperate the singular terms and the eigenvalues of the Birman-Schwinger operator attributed to each singular term. The different degrees of singular terms lead to different asymptotics for the eigenvalues of the Birman-Schwinger operator and hence, after resolving the implicit equation \eqref{eq:eigeval_schurcomplt_S_k}, for the eigenvalues of $H(\varepsilon)$. \\

\textbf{Step 5: Extracting eigenvalue asymptotics - first reference window}\\

Let us first consider $\lambda<0$ with 
\begin{align} \label{eq:ref_window1_non_int}
C_1 \leq \varepsilon f_0(\lambda) = \varepsilon |\lambda|^{-1} \leq C_2
\end{align}
for some $C_1,C_2 >0$ that we specify later. This means we have $C_2^{-1}\varepsilon \leq |\lambda| \leq C_1^{-1} \varepsilon$ and therefore 
\begin{align}
\varepsilon f_1(\lambda) &\asymp \varepsilon^{\alpha'}, \label{eq:eps_f1_ref_window1} \\
\varepsilon f_2(\lambda) &\asymp \varepsilon^{1-\alpha'}, \label{eq:eps_f2_ref_window1}
\end{align}
for all $\lambda$ in the reference window. Let $G$ be the orthogonal complement of $H_n$ in $H_{n+1}$, i.e.\ $G = H_{n+1} \ominus H_n$. The space $G$ is one-dimensional. Recall that $S_\varepsilon^-(\lambda)$ acts only on $H_{n+1}$ and not the full space $L^2(\R^2)$. We view the Schur complement $S_\varepsilon^-(\lambda)$ now as a block operator on $H_n \oplus G$, i.e.\
\begin{align}
S_\varepsilon^-(\lambda) = \begin{pmatrix}
S_{\varepsilon,11}^-(\lambda) & S_{\varepsilon,12}^-(\lambda)\\
S_{\varepsilon,21}^-(\lambda) & S_{\varepsilon,22}^-(\lambda)
\end{pmatrix} = \begin{pmatrix}
S_\varepsilon^-(\lambda)|_{H_{n} \to H_{n}} & S_\varepsilon^-(\lambda)|_{G \to H_{n}} \\
S_\varepsilon^-(\lambda)|_{H_{n} \to G} & S_\varepsilon^-(\lambda)|_{G \to G}
\end{pmatrix}. \nonumber
\end{align}
Our idea is to apply the SLFG Lemma again to the operator $S_\varepsilon^-(\lambda)$ in the above block form. We have for $\varepsilon$ small enough 
\begin{align} 
S_{\varepsilon}^-(\lambda)  &=  K_{\varepsilon,11}^-(\lambda) - K_{\varepsilon,12}^-(\lambda) (K_{\varepsilon,22}^-(\lambda))^{-1}K_{\varepsilon,21}^-(\lambda)  \nonumber \\
&=  (1 - \varepsilon v (P_-(A) - \lambda)^{-1} v)|_{H_n \to H_n}  - K_{\varepsilon,12}^-(\lambda) (K_{\varepsilon,22}^-(\lambda))^{-1}K_{\varepsilon,21}^-(\lambda)  \nonumber\\
&=  (1 - \varepsilon f_0(\lambda) vP_0^-v - \varepsilon f_1(\lambda) v\psi^- \langle v\psi^-, . \rangle - \varepsilon f_2(\lambda) v\varphi^- \langle v\varphi^-, . \rangle - \tilde{R}_\varepsilon(\lambda) )|_{H_n \to H_n}  \nonumber \\
& \hspace{6.3cm}- K_{\varepsilon,12}^-(\lambda) (K_{\varepsilon,22}^-(\lambda))^{-1}K_{\varepsilon,21}^-(\lambda).   \nonumber
\end{align}
Here, the operator $\tilde{R}_{\varepsilon}(\lambda)$ is the operator $\tilde{R}_{\varepsilon}(\lambda)= \varepsilon  \cdot ( v (P_-(A) - \lambda)^{-1} v -  S(\lambda))|_{H_n \to H_n}  $ with the singular terms $S(\lambda)$ from \eqref{eq:singular_part_nonint}. Since the $\varepsilon$-independent part of $\tilde{R}_{\varepsilon}(\lambda)$ is the uniformly bounded for bounded $|\lambda|$, it holds $\Vert \tilde{R}_{\varepsilon}(\lambda) \Vert \lesssim \varepsilon$. By the choice of $H_n$,
\begin{align}
1|_{G \to H_n} &= 1|_{H_n \to G} = 0, \nonumber\\
vP_0^-v|_{G \to H_n} &= vP_0^-v|_{H_n \to G}= vP_0^-v|_{G \to G} =0 , \nonumber \\
v\psi^- \langle v \psi^- , . \rangle |_{G \to H_n} &= v\psi^- \langle v \psi^- , . \rangle |_{H_n \to G} =  v\psi^- \langle v \psi^- , . \rangle |_{G \to G}= 0 ,\nonumber
\end{align}
so using \eqref{eq:K_eps122221_comb_nonint}, we find now for example
\begin{align}
S_{\varepsilon,12}^-(\lambda) &= S_{\varepsilon}^-(\lambda) |_{G \to H_n}  = - \varepsilon f_2(\lambda) (v\varphi^- \langle v\varphi^-, . \rangle)  |_{G \to H_n} + R_{\varepsilon,12}(\lambda) \nonumber
\end{align}
where 
\begin{align}
R_{\varepsilon,12}(\lambda) =( - \tilde{R}_{\varepsilon} (\lambda) - K_{\varepsilon,12}^-(\lambda) (K_{\varepsilon,22}^-(\lambda))^{-1}K_{\varepsilon,21}^-(\lambda) ) |_{G \to H_n} \nonumber
\end{align}
is a linear operator with $\Vert R_{\varepsilon,12}(\lambda)\Vert \lesssim \varepsilon + \varepsilon^2 \lesssim \varepsilon $ for all $\lambda <0 $ satisfying \eqref{eq:ref_window1_non_int}. Proceeding similarly with the other blocks of $S_{\varepsilon}^-(\lambda)$, we find that
\begin{align} 
S_{\varepsilon,11}^-(\lambda) &= (1 - \varepsilon f_0(\lambda) vP_0^-v -\varepsilon f_1(\lambda) v\psi^- \langle v\psi^-, . \rangle  - \varepsilon f_2(\lambda) (v\varphi^- \langle v\varphi^-, . \rangle)  |_{H_n \to H_n} + R_{\varepsilon,11}(\lambda)  , \label{eq:S_eps11_nont_int_window1_expansion} \\
S_{\varepsilon,12}^-(\lambda) &=  -\varepsilon f_2(\lambda) (v\varphi^- \langle v\varphi^-, . \rangle)  |_{G \to H_n} + R_{\varepsilon,12}(\lambda) , \label{eq:S_eps12_nont_int_window1_expansion}\\ 
S_{\varepsilon,21}^-(\lambda) &=  - \varepsilon f_2(\lambda) (v\varphi^- \langle v\varphi^-, . \rangle)  |_{H_n \to G} + R_{\varepsilon,21}(\lambda)  , \label{eq:S_eps21_nont_int_window1_expansion}\\
 S_{\varepsilon,22}^-(\lambda) &= (1 - \varepsilon f_2(\lambda) v\varphi^- \langle v\varphi^-, . \rangle )  |_{G \to G} + R_{\varepsilon,22}(\lambda) \label{eq:S_eps22_nont_int_window1_expansion}
\end{align}
where $R_{\varepsilon,ij}(\lambda)$ are some linear operators with $\Vert R_{\varepsilon,ij}(\lambda)\Vert \lesssim \varepsilon $ for all $\lambda <0 $ satisfying \eqref{eq:ref_window1_non_int}. Now, $ \varepsilon f_2(\lambda) \lesssim \varepsilon^{1-\alpha'}$ implies
\begin{align} 
\Vert S_{\varepsilon,12}^-(\lambda) \Vert& \lesssim \varepsilon^{1-\alpha'} + \varepsilon \lesssim \varepsilon^{1-\alpha'} , \label{eq:S_eps12_nonint}\\
\Vert S_{\varepsilon,21}^-(\lambda) \Vert& \lesssim \varepsilon^{1-\alpha'} + \varepsilon \lesssim \varepsilon^{1-\alpha'} , \label{eq:S_eps21_nonint}\\
\Vert S_{\varepsilon,22}^-(\lambda) -1|_{G \to G} \Vert& \lesssim \varepsilon^{1-\alpha'} + \varepsilon \lesssim \varepsilon^{1-\alpha'}. \label{eq:S_eps22_nonint}
\end{align}
From \eqref{eq:S_eps22_nonint} it follows that once $\varepsilon$ is small enough, $S_{\varepsilon,22}^-(\lambda)$ is invertible for all $\lambda<0$ in the reference window \eqref{eq:ref_window1_non_int} and then 
\begin{align}
\Vert (S_{\varepsilon,22}^-(\lambda))^{-1} -1|_{G \to G} \Vert&  \lesssim \varepsilon^{1-\alpha'}. \label{eq:S_eps22inv_nonint}
\end{align}
Applying the SLFG Lemma now reveals that $S_\varepsilon^-(\lambda)$ is invertible if and only if the Schur complement
\begin{align}
T_{\varepsilon}^-(\lambda) =  S_{\varepsilon,11}^-(\lambda) - S_{\varepsilon,12}^-(\lambda) (S_{\varepsilon,22}^-(\lambda))^{-1}S_{\varepsilon,21}^-(\lambda) \nonumber
\end{align}
is invertible in $H_n$. Here,
\begin{align}
\Vert S_{\varepsilon,12}^-(\lambda) (S_{\varepsilon,22}^-(\lambda))^{-1}S_{\varepsilon,21}^-(\lambda) \Vert \lesssim \varepsilon^{2(1-\alpha')} \label{eq:S_eps122221_eps_bound}
\end{align}
due to \eqref{eq:S_eps12_nonint}, \eqref{eq:S_eps21_nonint} and \eqref{eq:S_eps22inv_nonint}. 

What we have achieved now is that we reduced the Schur complement $S_{\varepsilon}^-(\lambda)$ by one dimension to $T_{\varepsilon}^-(\lambda)$ by taking another Schur complement. The dimension we got rid of by restricting our view to the reference window is the dimension spanned by $v\psi_n$ which appears only in the lower order singular term $\varepsilon f_2(\lambda) (v\varphi^- \langle v\varphi^-, . \rangle) $ of the Birman-Schwinger operator. When taking this second Schur complement, we pay with another term added of order $\varepsilon^{2(1-\alpha')}$.

We are now ready to find asymptotic expansions of negative eigenvalues of $H(\varepsilon)$ within our chosen reference window. The Schur complement $T_{\varepsilon}^-(\lambda)$ is not invertible if and only if one of its eigenvalues is zero. We denote the eigenvalues of $T_{\varepsilon}^-(\lambda)$ by $\mu_k (T_{\varepsilon}^-(\lambda))$, $k=0, ..., n-1$. We show that each equation $\mu_k (T_{\varepsilon}^-(\lambda))=0$ has at least one solution, if we choose the frame of the reference window, given by $C_1$, $C_2$, appropriately.

The Schur complement $T_\varepsilon^-(\lambda)$ can be written as
\begin{align}  
T_{\varepsilon}^-(\lambda) &=( 1 - \varepsilon f_0(\lambda) vP_0^-v -\varepsilon f_1(\lambda) v\psi^- \langle v\psi^-, . \rangle - \varepsilon f_2(\lambda) v\varphi^- \langle v\varphi^-, . \rangle  ) |_{H_n \to H_n} \nonumber \\
& \qquad\qquad \qquad  + R_{\varepsilon,11}(\lambda) - S_{\varepsilon,12}^-(\lambda) (S_{\varepsilon,22}^-(\lambda))^{-1}S_{\varepsilon,21}^-(\lambda)  \nonumber
\end{align}
where $R_{\varepsilon,11}(\lambda)$ is some linear operator with $\Vert R_{\varepsilon,11}(\lambda)\Vert \lesssim \varepsilon $ for all $\lambda <0 $ satisfying \eqref{eq:ref_window1_non_int}. Because of \eqref{eq:eps_f1_ref_window1}, \eqref{eq:eps_f2_ref_window1} and \eqref{eq:S_eps122221_eps_bound} we see that for $\lambda$ in the given reference window
\begin{align} \label{eq:T_eps_minus_const_estimate}
\Vert T_\varepsilon^-(\lambda) - (1 - \varepsilon f_0(\lambda) vP_0^-v)|_{H_n \to H_n} \Vert \lesssim \varepsilon^{\alpha'} +\varepsilon^{1-\alpha'} +\varepsilon+ \varepsilon^{2(1-\alpha')} \lesssim \varepsilon^{\min \{ \alpha',1-\alpha'\}}.
\end{align}

Let us briefly focus on the operator $(vP_0^- v) |_{H_n \to H_n}$. It is clearly finite rank and self-adjoint. But it is also positive. It is easily seen to be non-negative, because for any $\phi \in H_n $ 
\begin{align}
\langle \phi , (vP_0^- v) \phi \rangle = \Vert P_0^- v \phi \Vert^2_{L^2(\R^2)} \geq 0. \nonumber
\end{align}
Now suppose there is some $\phi \in H_n $ with $\langle \phi , (vP_0^- v) \phi \rangle = \Vert P_0^- v \phi \Vert^2_{L^2(\R^2)} = 0$. Then $v\phi$ must be orthogonal to all zero eigenstates $\{\psi_k^-\}_{k=0}^{n-1}$. But this implies 
\begin{align}
\langle \phi, v\psi_k^-\rangle = \langle v\phi, \psi_k^-\rangle=0 \nonumber
\end{align}
for any $k=0,...,n-1$, hence $\phi \in H_n^\perp$ and so $\phi=0$. This shows that $(vP_0^- v) |_{H_n \to H_n}$ is positive and thus it has only positive eigenvalues. Let $\mu_k$, $k=0,...,n-1$, denote these eigenvalues. 

We turn back to the eigenvalues of the Schur complement $T_\varepsilon^-(\lambda)$. Suppose $C_1,C_2$ are chosen such that $\{ \mu_k^{-1} \}_{k=0}^{n-1} \subset (C_1+\delta,C_2-\delta)$ for some small $\delta>0$. By \eqref{eq:T_eps_minus_const_estimate}, 
\begin{align} \label{eq:mu_k_T_eps_estimate}
\left| \mu_k(T_\varepsilon^-(\lambda)) - (1 - \varepsilon f_0(\lambda) \mu_k) \right| \lesssim \varepsilon^{\min \{ \alpha',1-\alpha'\}}
\end{align}
for all $\lambda$ in the reference window, if $\varepsilon$ is small enough. It follows that $\mu_k(T_\varepsilon^-(\lambda)) > 0$ for small enough $\varepsilon$ if $\lambda$ is chosen such that $\varepsilon f_0(\lambda) \leq  \mu_k^{-1}-\delta$ and $\mu_k(T_\varepsilon^-(\lambda')) < 0$ for small enough $\varepsilon$ if $\lambda'$ is chosen such that $\varepsilon f_0(\lambda') \geq \mu_k^{-1} + \delta$. Since $\mu_k(T_\varepsilon^-(\lambda)) $ is continuous in $\lambda$, there exists $\lambda_k(\varepsilon)$ such that $\mu_k(T_\varepsilon^-(\lambda_k(\varepsilon)))  =0 $ by the intermediate value theorem. Finally, for this $\lambda_k(\varepsilon)$ holds according to \eqref{eq:mu_k_T_eps_estimate}
\begin{align}
|1- \varepsilon f_0(\lambda_k(\varepsilon)) \mu_k|  \lesssim \varepsilon^{\min \{ \alpha',1-\alpha'\}}  \nonumber
\end{align}
which yields
\begin{align}
\lambda_k(\varepsilon) = - \mu_k \, \varepsilon \, (1+ O(\varepsilon^{\min \{ \alpha',1-\alpha'\}} )) \nonumber
\end{align}
as $\varepsilon \searrow 0$.\\

\textbf{Step 6: Extracting eigenvalue asymptotics - second reference window}\\

Let us now consider $\lambda<0$ with 
\begin{align} \label{eq:ref_window2_non_int}
C_1 \leq \varepsilon |\lambda|^{-\alpha'}  \leq C_2
\end{align}
for some yet-to-be-specified $C_1,C_2 >0$. This means we have $(C_2^{-1}\varepsilon)^\frac{1}{\alpha'} \leq |\lambda| \leq (C_1^{-1} \varepsilon)^\frac{1}{\alpha'}$ and therefore 
\begin{align}
 \varepsilon f_0(\lambda) &\asymp \varepsilon^{1-\frac{1}{\alpha'}}, \label{eq:eps_f0_ref_window2} \\
 \varepsilon f_1(\lambda) &\asymp \varepsilon^{2-\frac{1}{\alpha'}}. \label{eq:eps_f1_ref_window2} \\
\varepsilon f_2(\lambda) &\asymp 1. \label{eq:eps_f2_ref_window2}
\end{align}
We decompose the Schur complement $S_\varepsilon^-(\lambda)$ as in the previous step into 
\begin{align}
S_\varepsilon^-(\lambda) = \begin{pmatrix}
S_{\varepsilon,11}^-(\lambda) & S_{\varepsilon,12}^-(\lambda)\\
S_{\varepsilon,21}^-(\lambda) & S_{\varepsilon,22}^-(\lambda)
\end{pmatrix} = \begin{pmatrix}
S_\varepsilon^-(\lambda)|_{H_{n} \to H_{n}} & S_\varepsilon^-(\lambda)|_{G \to H_{n}} \\
S_\varepsilon^-(\lambda)|_{H_{n} \to G} & S_\varepsilon^-(\lambda)|_{G \to G}
\end{pmatrix}, \nonumber
\end{align}
where again $G=H_{n+1} \ominus H_n$. As in the case of the first reference window, the blocks then satisfy for $\varepsilon$ small enough equations \eqref{eq:S_eps11_nont_int_window1_expansion}-\eqref{eq:S_eps22_nont_int_window1_expansion} with some linear operators $R_{\varepsilon,ij}(\lambda)$ with $\Vert R_{\varepsilon,ij}(\lambda)\Vert \lesssim \varepsilon $ for all $\lambda <0 $ satisfying \eqref{eq:ref_window2_non_int}. We apply the SLFG Lemma again, but this time we exchange the roles of $S_{11}^-(\lambda)$ and $S_{22}^-(\lambda)$. The choice of the reference window implies this time
\begin{align} 
\Vert (\varepsilon f_0(\lambda ))^{-1} S_{\varepsilon,11}^-(\lambda) -(vP_0^-v)|_{H_n \to H_n} \Vert& \lesssim \varepsilon^{\frac{1}{\alpha'}-1} , \label{eq:S_eps11_nonint_window2}  \\
\Vert S_{\varepsilon,12}^-(\lambda) \Vert& \lesssim 1, \label{eq:S_eps12_nonint_window2}\\
\Vert S_{\varepsilon,21}^-(\lambda) \Vert& \lesssim 1. \label{eq:S_eps21_nonint_window2}
\end{align}
Since $(vP_0^-v)|_{H_n \to H_n} $ is positive and hence invertible, it follows that once $\varepsilon$ is small enough, $S_{\varepsilon,11}^-(\lambda)$ is invertible for all $\lambda<0$ in the reference window \eqref{eq:ref_window2_non_int} and 
\begin{align}
\Vert (S_{\varepsilon,11}^-(\lambda))^{-1} \Vert &  \asymp \varepsilon^{\frac{1}{\alpha'}-1}. \label{eq:S_eps22inv_nonint_window2}
\end{align}
Applying the SLFG Lemma now shows that $S_\varepsilon^-$ is invertible if and only if the Schur complement
\begin{align} \label{eq:W_eps_def}
W_{\varepsilon}^-(\lambda) =  S_{\varepsilon,22}^-(\lambda) - S_{\varepsilon,21}^-(\lambda) (S_{\varepsilon,11}^-(\lambda))^{-1}S_{\varepsilon,12}^-(\lambda)
\end{align}
is invertible in $G$. Here,
\begin{align}
\Vert S_{\varepsilon,21}^-(\lambda) (S_{\varepsilon,11}^-(\lambda))^{-1}S_{\varepsilon,12}^-(\lambda) \Vert \lesssim \varepsilon^{\frac{1}{\alpha'}-1} \label{eq:S_eps211112_eps_bound}
\end{align}
due to \eqref{eq:S_eps12_nonint_window2}, \eqref{eq:S_eps21_nonint_window2} and \eqref{eq:S_eps22inv_nonint_window2}. This time, $W_{\varepsilon}^-(\lambda) $ is a rank one operator and can be written as
\begin{align}
W_{\varepsilon}^-(\lambda) = w_\varepsilon^-(\lambda) \, \varphi_n \langle \varphi_n , . \rangle  \nonumber
\end{align}
where $\varphi_n$ is an $L^2$-normalized state that spans $G$. The scalar function $w_\varepsilon^-(\lambda)$ is given by
\begin{align}
w_\varepsilon^-(\lambda) = \langle \varphi_n , W_{\varepsilon}^-(\lambda)\varphi_n  \rangle. \nonumber
\end{align}
The Schur complement $W_{\varepsilon}^-(\lambda)$ is not invertible if and only if $w_\varepsilon^-(\lambda) = 0$. 

We argue that $w_\varepsilon^-(\lambda) = 0$ has at least one solution for appropriately chosen reference window frame. Let
\begin{align}
\nu_n =\langle v \varphi^-, Q(v \varphi^-)\rangle  = \Vert Q(v \varphi^-) \Vert_{L^2(\R^2)}^{2} \nonumber
\end{align}
where $Q$ denotes the orthogonal projection onto $G$ (we may now understand $G$ as a subspace of $L^2(\R^2)$). We remark that $\nu_n>0$ since $v\varphi^- \notin H_n$. Let $C_1,C_2>0$ be such that $\nu_n^{-1}  \in (C_1+\delta,C_2-\delta)$ for some small $\delta >0$. Then, from \eqref{eq:W_eps_def}, \eqref{eq:S_eps211112_eps_bound} and \eqref{eq:S_eps22_nont_int_window1_expansion} follows
\begin{align}
|w_\varepsilon^-(\lambda) - (1-\varepsilon f_2(\lambda) \nu_n )| \lesssim \varepsilon + \varepsilon^{\frac{1}{\alpha'}-1} \lesssim \varepsilon^{\min\{1,\frac{1}{\alpha'}-1\}}. \label{eq:w_eps_estimate}
\end{align} 
We now always find $\lambda$ in our reference window, such that $\varepsilon f_2(\lambda)\leq \nu_n^{-1}-\delta$. For this $\lambda$ we conclude $w_\varepsilon^-(\lambda)>0$, if $\varepsilon$ is sufficiently small. We also always find $\lambda'$ in our reference window, such that $\varepsilon f_2(\lambda')\geq \nu_n^{-1}+\delta$ which then yields $w_\varepsilon^-(\lambda')<0$. Since $w_\varepsilon^-(\lambda) $ is continuous in $\lambda$, we now find $\lambda_n(\varepsilon)$ such that $w_\varepsilon^-(\lambda_n(\varepsilon))   =0 $. This solution satisfies 
\begin{align}
|1- \varepsilon f_2(\lambda_k(\varepsilon)) \nu_n|  \lesssim \varepsilon^{\min\{1,\frac{1}{\alpha'}-1\}} \nonumber
\end{align}
due to \eqref{eq:w_eps_estimate}, which implies with the definition of $f_2(\lambda)$ that
\begin{align}
\lambda_n(\varepsilon) = - \left( \frac{4^{\alpha'-1} \Gamma(\alpha') }{\pi \Gamma(1-\alpha')}\, \nu_n \right)^\frac{1}{\alpha'}  \varepsilon^\frac{1}{\alpha'}   \, (1+ O(\varepsilon^{\min\{1,\frac{1}{\alpha'}-1\}} ))  \nonumber
\end{align}
as $\varepsilon \searrow 0$.

Note that if $0< \alpha < 1$, then $n=0$, $\alpha' = \alpha$ and $H_n = \{ 0 \}$. The first Schur complement $S_\varepsilon^-(\lambda)$ is already a rank one operator and doesn't need to be decomposed further. We can then skip the first reference window and directly work with $W_\varepsilon^-(\lambda) = S_{\varepsilon}^-(\lambda)$ in the second reference window. Therefore, no error term of order $\varepsilon^{\frac{1}{\alpha'}-1}$ appears and one eventually finds that $\lambda_0(\varepsilon)$ satisfies
\begin{align}
\lambda_0(\varepsilon) = - \left( \frac{4^{\alpha'-1} \Gamma(\alpha') }{\pi \Gamma(1-\alpha')}\, \nu_n \right)^\frac{1}{\alpha'}  \varepsilon^\frac{1}{\alpha'}  \, (1+ O(\varepsilon ))  \nonumber
\end{align}
as $\varepsilon \searrow 0$.
\begin{flushright}
$\square$
\end{flushright}

\subsection{Integer flux - Proof of Theorem \ref{thm:expansion_int}}

For the case of integer flux, we basically repeat the procedure we applied in the non-integer case, i.e.\ choosing suitable reference windows and taking iterated Schur complements. Of course, due the resolvent expansion that we now draw from Theorem \ref{thm:resolvent_int}, the explicit first order terms and the bounds on the second order terms will change accordingly. \\

\textbf{Step 1: Preliminaries}\\

Let $\alpha > 0$, $\alpha \in \mathbb{Z}$ and $n=\alpha-1$. Again, the Aharonov-Casher states $\{\psi_k^- \}_{k=0}^{n-1}$ span the zero eigenspace of $P_-(A)$. Additionally, there are two virtual bound states at zero, namely the generalized Aharonov-Casher states $\psi_n^-$ and $\psi_{n+1}^-$. All (generalized) Aharonov-Casher states states are bounded and linearly independent, so $v\psi_k^- \in L^2(\R^2)$ for all $k=0,...,n+1$ and all $v\psi_k^-$ are linearly independent, given $V$ satisfies Assumption \ref{assump:potentialV}.

Set now
\begin{align}
H_{n} &= \operatorname{span}\{ v\psi_k^- \}_{k=0}^{n-1}, \nonumber \\
H_{n+1} &= \operatorname{span}(\{ v\psi_k^- \}_{k=0}^{n-1} \cup \{ v\varphi_2^-\}) ,\nonumber \\
H_{n+2} &= \operatorname{span}(\{ v\psi_k^- \}_{k=0}^{n-1} \cup \{ v\varphi_2^-, v\varphi_1^-\})  \nonumber
\end{align}
where $\varphi_1^-$, $\varphi_2^-$ are the virtual bound states from Theorem \ref{thm:resolvent_int}.
The spaces $H_{n}$, $H_{n+1}$ and $H_{n+2}$ are $n$-, $(n+1)$- resp.~$(n+2)$-dimensional subspaces of $L^2(\R^2)$. Further define for $\lambda <0$, $|\lambda|$ small enough,
\begin{align}
f_0(\lambda)&= -\lambda^{-1} = |\lambda|^{-1},\nonumber \\
f_1(\lambda)&=\frac{1}{\pi \lambda (\log |\lambda| +m)} = - \frac{1}{\pi |\lambda| (\log |\lambda| +m)},\nonumber \\
f_2(\lambda)&= - \log|\lambda|. \nonumber
\end{align}
with $m$ as it appears in Theorem \ref{thm:resolvent_int}.

Theorem \ref{thm:resolvent_int} then implies
\begin{align}
v(P_-(A)-\lambda)^{-1}v= f_0(\lambda) vP_0^- v+ f_1(\lambda)v\Pi_{22}v +f_2(\lambda) vKv + O(1) \label{eq:vPv_int_expansion}
\end{align}
in $\mathcal{L}(L^2(\R^2))$ as $\lambda \nearrow 0$, with $\Pi_{22}$ and $K$ as defined in \eqref{eq:definition_Pi_op} and \eqref{eq:definition_K_op}. Notice that all terms on the right-hand side of \eqref{eq:vPv_int_expansion} are self-adjoint. \\

\textbf{Step 2: Bounding the number of negative eigenvalues of $H(\varepsilon)$}\\

Similarly as in the previous cases $\alpha = 0$ and $\alpha \in\mathbb{R} \setminus \mathbb{Z}$, the asymptotic equation \eqref{eq:vPv_int_expansion} allows to argue that $H(\varepsilon)$ has no more than $n+2$ negative eigenvalues. Again, $P_+(A)-\varepsilon V$ contributes no negative eigenvalues for small enough $\varepsilon$, so one needs to argue that $P_-(A)-\varepsilon V$ has at most $n+2$ negative eigenvalues for small enough $\varepsilon$. The reason for that is that the operator 
\begin{align}
S(\lambda)=f_0(\lambda) vP_0^-v +f_1(\lambda) v\psi^- \langle v\psi^-, . \rangle + f_2(\lambda) v\varphi^- \langle v\varphi^-, . \rangle, \nonumber
\end{align}
which collects all terms on the right hand side of \eqref{eq:vPv_int_expansion} that are singular as $\lambda \to 0$, is rank $n+2$. Denote by $\lambda_k^-(\varepsilon)$, $k=0,1,...$,  the negative eigenvalues of $P_-(A) - \varepsilon V$ and $\mu_k^-(\lambda)$, $k=0,1,...$, the non-negative eigenvalues of $v(P_{-}(A)-\lambda)^{-1}v$. Observing that $f_0(\lambda)$, $f_1(\lambda)$ and $f_2(\lambda)$ are positive for $|\lambda|$ small enough, one can eventually show as in the previous cases that $
|\varepsilon \mu_k^-(\lambda) | \leq 1 $ for $k> n+1$ and $|\lambda|, \varepsilon $ small enough. Again, one concludes for $\varepsilon $ small enough
\begin{align}
\#\{ \lambda_k^\pm(\varepsilon)  \} &= \lim_{\lambda \nearrow 0} \# \{ \lambda_k(\varepsilon) \, : \, \lambda_k(\varepsilon)< \lambda  \} = \lim_{\lambda \nearrow 0} \# \{ \varepsilon\mu_k^\pm(\lambda) \, : \, \varepsilon \mu_k^\pm(\lambda)>1 \} \nonumber \\
& \leq  \lim_{\lambda \nearrow 0}  \# \{ \varepsilon\mu_0^\pm(\lambda),...,\varepsilon\mu_{n+1}^\pm(\lambda) \} = n+2 \nonumber
\end{align}
by the Birman-Schwinger principle. This shows that $P_-(A) - \varepsilon V$ and hence $H(\varepsilon)$ has at most $n+2$ negative eigenvalues.\\

\textbf{Step 3: Reduction to a finite rank operator}\\

To extract eigenvalue asymptotics, we consider 
\begin{align}
K_\varepsilon^-(\lambda) = 1 - \varepsilon v(P_-(A)-\lambda)^{-1}v \nonumber
\end{align}
as before and check its invertibility. We decompose this operator into
\begin{align}
K_{\varepsilon}^-(\lambda)= \begin{pmatrix}
K_{\varepsilon,11}^-(\lambda) & K_{\varepsilon,12}^-(\lambda)\\[5pt]
K_{\varepsilon,21}^-(\lambda) & K_{\varepsilon,22}^-(\lambda)
\end{pmatrix} = \begin{pmatrix}
K_\varepsilon^-(\lambda)|_{H_{n+2} \to H_{n+2}} & K_\varepsilon^-(\lambda)|_{H_{n+2}^\perp \to H_{n+2}} \\
K_\varepsilon^-(\lambda)|_{H_{n+2} \to H_{n+2}^\perp} & K_\varepsilon^-(\lambda)|_{H_{n+2}^\perp \to H_{n+2}^\perp}
\end{pmatrix}. \nonumber
\end{align}
Following the arguments in the previous section, one finds that for $\varepsilon$ small enough, $K_{\varepsilon}^-(\lambda)$ is invertible if and only if its Schur complement
\begin{align}\label{eq:S_eps_int}
S_{\varepsilon}^-(\lambda) = K_{\varepsilon,11}^-(\lambda) - K_{\varepsilon,12}^-(\lambda) (K_{\varepsilon,22}^-(\lambda))^{-1}K_{\varepsilon,21}^-(\lambda)
\end{align}
is invertible in $H_{n+2}$. Again, one finds 
\begin{align} \label{eq:K_eps122221_comb_int}
\Vert K_{\varepsilon,12}^-(\lambda) (K_{\varepsilon,22}^-(\lambda))^{-1}K_{\varepsilon,21}^-(\lambda) \Vert \lesssim \varepsilon^2 
\end{align}
and any solution $\lambda$ to one of the $n+2$ equations 
\begin{align} 
\mu_k (S_{\varepsilon}^-(\lambda))  = 0, \qquad k=0,...,n+1, \nonumber
\end{align}
where $\mu_k (S_{\varepsilon}^-(\lambda)) $ denote the eigenvalues of $S_{\varepsilon}^-(\lambda) $, yields a negative eigenvalue of $H(\varepsilon)$.\\

\textbf{Step 4: Showing existence of negative eigenvalues of $H(\varepsilon)$}\\

This step is completely analogous to the same step in the non-integer flux case. Each solution of one of the equations $ \mu_k (S_{\varepsilon}^-(\lambda)) $, $k=0,..., n+1$, yields precisely one eigenvalue $\lambda_k(\varepsilon)$ of $H(\varepsilon)$. We find that if $\varepsilon$ is small enough, $H(\varepsilon)$ has precisely $n+2$ eigenvalues, counting multiplicities.\\ 

Next, we compute the eigenvalue asymptotics by investigating the finite dimensional operator $S_{\varepsilon}^-(\lambda)$ and asking when it is not invertible in $H_{n+2}$. This time we need to apply three different reference windows.\\

\textbf{Step 5: Extracting eigenvalue asymptotics - first reference window}\\

As in the case of non-integer flux, we first consider $\lambda<0$ with 
\begin{align} \label{eq:ref_window1_int}
C_1 \leq \varepsilon f_0(\lambda)  = \varepsilon |\lambda|^{-1} \leq C_2
\end{align}
for some $C_1,C_2 >0$ that we specify later. It follows
\begin{align}
\varepsilon f_1(\lambda) &\asymp |\log\varepsilon|^{-1}, \label{eq:eps_f1_ref_window1_int} \\
\varepsilon f_2(\lambda) &\asymp |\varepsilon \log\varepsilon |. \label{eq:eps_f2_ref_window1_int}
\end{align}

Let $G_{2}$ denote the orthogonal complement of $H_n$ in $H_{n+2}$, i.e.\ $G_{2} = H_{n+2} \ominus H_n$. The space $G_{2}$ is two-dimensional. Recall that $S_\varepsilon^-(\lambda)$ acts on $H_{n+2}$. We decompose the Schur complement $S_\varepsilon^-(\lambda)$ into 
\begin{align}
S_\varepsilon^-(\lambda) = \begin{pmatrix}
S_{\varepsilon,11}^-(\lambda) & S_{\varepsilon,12}^-(\lambda)\\
S_{\varepsilon,21}^-(\lambda) & S_{\varepsilon,22}^-(\lambda)
\end{pmatrix} = \begin{pmatrix}
S_\varepsilon^-(\lambda)|_{H_{n} \to H_{n}} & S_\varepsilon^-(\lambda)|_{G_{2} \to H_{n}} \\
S_\varepsilon^-(\lambda)|_{H_{n} \to G_{2}} & S_\varepsilon^-(\lambda)|_{G_{2} \to G_{2}}
\end{pmatrix}. \nonumber
\end{align}
Using \eqref{eq:vPv_int_expansion} and \eqref{eq:K_eps122221_comb_int}, we find for $\varepsilon$ small enough,
\begin{align} 
S_{\varepsilon,12}^-(\lambda)  &=  (  K_{\varepsilon,11}^-(\lambda) - K_{\varepsilon,12}^-(\lambda) (K_{\varepsilon,22}^-(\lambda))^{-1}K_{\varepsilon,21}^-(\lambda) )|_{G_{2} \to H_n}  \nonumber \\
& = (-\varepsilon f_1(\lambda) v\Pi_{22}v - \varepsilon f_2(\lambda) vKv) |_{G_2 \to H_n} + R_{\varepsilon,12}(\lambda), \nonumber \\
S_{\varepsilon,21}^-(\lambda) &=   (-\varepsilon f_1(\lambda) v\Pi_{22}v - \varepsilon f_2(\lambda) vKv)  |_{H_n \to G_{2}} + R_{\varepsilon,21}(\lambda), \nonumber \\
S_{\varepsilon,22}^-(\lambda) &= (1-  \varepsilon f_1(\lambda) v\Pi_{22}v - \varepsilon f_2(\lambda) vKv) |_{G_{2} \to G_{2}} + R_{\varepsilon,22}(\lambda)  \nonumber
\end{align}
where $R_{\varepsilon,ij}(\lambda)$ are some linear operators with $\Vert R_{\varepsilon,ij}(\lambda)\Vert \lesssim \varepsilon $ for all $\lambda <0 $ satisfying \eqref{eq:ref_window1_int}. Now, \eqref{eq:eps_f1_ref_window1_int} and \eqref{eq:eps_f2_ref_window1_int} imply
\begin{align} 
\Vert S_{\varepsilon,12}^-(\lambda) \Vert& \lesssim  |\log\varepsilon|^{-1} + |\varepsilon \log\varepsilon | + \varepsilon \lesssim |\log\varepsilon|^{-1}, \label{eq:S_eps12_int}\\
\Vert S_{\varepsilon,21}^-(\lambda) \Vert& \lesssim  |\log\varepsilon|^{-1} +|\varepsilon \log\varepsilon | + \varepsilon \lesssim |\log\varepsilon|^{-1}, \label{eq:S_eps21_int}\\
\Vert S_{\varepsilon,22}^-(\lambda) -1|_{G_{2} \to G_{2} } \Vert& \lesssim  |\log\varepsilon|^{-1} + |\varepsilon \log\varepsilon | + \varepsilon \lesssim |\log\varepsilon|^{-1}. \label{eq:S_eps22_int}
\end{align}
From \eqref{eq:S_eps22_int} it follows that once $\varepsilon$ is small enough, $S_{\varepsilon,22}^-(\lambda)$ is invertible for all $\lambda<0$ in the reference window \eqref{eq:ref_window1_int} and then 
\begin{align}
\Vert (S_{\varepsilon,22}^-(\lambda))^{-1} -1|_{G_{2} \to G_{2}} \Vert&  \lesssim  |\log\varepsilon|^{-1}. \label{eq:S_eps22inv_int}
\end{align}
Applying the SLFG Lemma, we see that $S_\varepsilon^-(\lambda)$ is invertible if and only if the Schur complement
\begin{align}
T_{\varepsilon}^-(\lambda) =  S_{\varepsilon,11}^-(\lambda) - S_{\varepsilon,12}^-(\lambda) (S_{\varepsilon,22}^-(\lambda))^{-1}S_{\varepsilon,21}^-(\lambda) \nonumber
\end{align}
is invertible in $H_n$. Here, the estimate
\begin{align}
\Vert S_{\varepsilon,12}^-(\lambda) (S_{\varepsilon,22}^-(\lambda))^{-1}S_{\varepsilon,21}^-(\lambda) \Vert \lesssim (\log\varepsilon)^{-2} \label{eq:S_eps122221_eps_bound_int}
\end{align}
holds, due to \eqref{eq:S_eps12_int}, \eqref{eq:S_eps21_int} and \eqref{eq:S_eps22inv_int}. 

As before, let us denote eigenvalues of $T_{\varepsilon}^-(\lambda)$ by $\mu_k (T_{\varepsilon}^-(\lambda))$, $k=0, ..., n-1$, and assume them to be sorted in non-decreasing order. The Schur complement $T_{\varepsilon}^-(\lambda)$ is not invertible if and only if $\mu_k (T_{\varepsilon}^-(\lambda))=0$ for some $k$. As in the case of non-integer flux, it may be argued that each equation $\mu_k (T_{\varepsilon}^-(\lambda))=0$ has at least one solution $\lambda_k(\varepsilon)$, if the frame of the reference window is chosen appropriately. If $\{ \mu_k \}_{k=0}^{n-1}$ denote again the (positive) eigenvalues of $(vP_0^- v) |_{H_n \to H_n}$, we eventually find
\begin{align} \label{eq:mu_k_T_eps_estimate_int}
\left| \mu_k(T_\varepsilon^-(\lambda)) - (1 - \varepsilon f_0(\lambda) \mu_k) \right| \lesssim  |\log\varepsilon |^{-1} +|\varepsilon \log\varepsilon| + \varepsilon  + (\log\varepsilon)^{-2} \lesssim |\log\varepsilon |^{-1}
\end{align} 
for all $\lambda$ in the reference window, if $\varepsilon$ is small enough. It follows that 
\begin{align}
|1- \varepsilon f_0(\lambda_k(\varepsilon)) \mu_k|  \lesssim |\log\varepsilon |^{-1}  \nonumber
\end{align}
which yields
\begin{align}
\lambda_k(\varepsilon) = - \mu_k \, \varepsilon  \left(1+ O\left(|\log \varepsilon|^{-1} \right) \right) \nonumber
\end{align}
as $\varepsilon \searrow 0$.\\

\textbf{Step 6: Extracting eigenvalue asymptotics - second reference window}\\

Now consider $\lambda<0$ with 
\begin{align} \label{eq:ref_window2_int}
C_1 \leq \varepsilon f_1(\lambda) \leq C_2
\end{align}
for some $C_1,C_2 >0$. It follows
\begin{align}
\varepsilon f_0(\lambda) &\asymp |\log\varepsilon|, \label{eq:eps_f1_ref_window2_int} \\
\varepsilon f_2(\lambda) &\asymp |\varepsilon \log\varepsilon |. \label{eq:eps_f2_ref_window2_int}
\end{align}

We decompose the Schur complement $S_\varepsilon^-(\lambda)$ again into 
\begin{align}
S_\varepsilon^-(\lambda) = \begin{pmatrix}
S_{\varepsilon,11}^-(\lambda) & S_{\varepsilon,12}^-(\lambda)\\
S_{\varepsilon,21}^-(\lambda) & S_{\varepsilon,22}^-(\lambda)
\end{pmatrix} = \begin{pmatrix}
S_\varepsilon^-(\lambda)|_{H_{n} \to H_{n}} & S_\varepsilon^-(\lambda)|_{G_{2} \to H_{n}} \\
S_\varepsilon^-(\lambda)|_{H_{n} \to G_{2}} & S_\varepsilon^-(\lambda)|_{G_{2} \to G_{2}}
\end{pmatrix}. \nonumber
\end{align}
with $G_{2}$ being the orthogonal complement of $H_n$ in $H_{n+2}$. With \eqref{eq:vPv_int_expansion} and \eqref{eq:K_eps122221_comb_int}, we find for $\varepsilon$ small enough,
\begin{align} 
S_{\varepsilon,11}^-(\lambda) &= (1  -  \varepsilon f_0(\lambda) vP_0^-v - \varepsilon f_2(\lambda) vKv) |_{H_n \to H_n} + R_{\varepsilon,11}(\lambda), \nonumber  \\
S_{\varepsilon,12}^-(\lambda)  & = (-\varepsilon f_1(\lambda) v\Pi_{22}v - \varepsilon f_2(\lambda) vKv) |_{G_{2} \to H_n} + R_{\varepsilon,12}(\lambda), \nonumber \\
S_{\varepsilon,21}^-(\lambda) &=   (-\varepsilon f_1(\lambda) v\Pi_{22}v- \varepsilon f_2(\lambda) vKv)  |_{H_n \to G_{2}} + R_{\varepsilon,21}(\lambda)  \nonumber
\end{align}
where $R_{\varepsilon,ij}(\lambda)$ are some linear operators with $\Vert R_{\varepsilon,ij}(\lambda)\Vert \lesssim \varepsilon $ for all $\lambda <0 $ satisfying \eqref{eq:ref_window2_int}. This time, \eqref{eq:eps_f1_ref_window2_int} and \eqref{eq:eps_f2_ref_window2_int} imply 
\begin{align}
\Vert (\varepsilon f_0(\lambda))^{-1} S_{\varepsilon,11}^-(\lambda) - vP_0^- v|_{H_n \to H_n} \Vert& \lesssim |\log\varepsilon|^{-1} (1 + \varepsilon| \log \varepsilon| +   \varepsilon )\lesssim |\log\varepsilon|^{-1}, \label{eq:S_eps11_window2_int} \\
\Vert S_{\varepsilon,12}^-(\lambda) \Vert& \lesssim 1, \label{eq:S_eps12_window2_int}\\
\Vert S_{\varepsilon,21}^-(\lambda) \Vert& \lesssim 1. \label{eq:S_eps21_window2_int}
\end{align}
Because $(vP_0^- v)|_{H_n \to H_n}$ is positive and thus invertible, it follows from \eqref{eq:S_eps11_window2_int} that once $\varepsilon$ is small enough, $S_{\varepsilon,11}^-(\lambda)$ is invertible for all $\lambda<0$ in the reference window \eqref{eq:ref_window2_int} and then 
\begin{align}
\Vert (S_{\varepsilon,11}^-(\lambda))^{-1} \Vert&  \lesssim |\log\varepsilon|^{-1}. \label{eq:S_eps11inv_window2_int}
\end{align}
Using the SLFG Lemma, we infer that $S_\varepsilon^-(\lambda)$ is invertible if and only if the Schur complement
\begin{align}
W_{\varepsilon}^-(\lambda) =  S_{\varepsilon,22}^-(\lambda) - S_{\varepsilon,21}^-(\lambda) (S_{\varepsilon,11}^-(\lambda))^{-1}S_{\varepsilon,12}^-(\lambda) \nonumber
\end{align}
is invertible in $G_{2}$. Here, 
\begin{align}
\Vert S_{\varepsilon,21}^-(\lambda) (S_{\varepsilon,11}^-(\lambda))^{-1}S_{\varepsilon,12}^-(\lambda) \Vert \lesssim |\log\varepsilon|^{-1} \label{eq:S_eps211112_eps_bound_window2_int}
\end{align}
due to \eqref{eq:S_eps12_window2_int},  \eqref{eq:S_eps21_window2_int} and \eqref{eq:S_eps11inv_window2_int}.  

The operator $W_{\varepsilon}^-(\lambda)$ acts on the two-dimensional space $G_{2}$. We can write $W_{\varepsilon}^-(\lambda)$ as
\begin{align}
W_{\varepsilon}^-(\lambda) = (1  -  \varepsilon f_1(\lambda) v\Pi_{22}v - \varepsilon f_2(\lambda) vKv) |_{G_{2} \to G_{2}} + R_{\varepsilon,22}(\lambda) - S_{\varepsilon,21}^-(\lambda) (S_{\varepsilon,11}^-(\lambda))^{-1}S_{\varepsilon,12}^-(\lambda)  \nonumber
\end{align}
where $R_{\varepsilon,22}(\lambda)$ is some linear operator with $\Vert R_{\varepsilon,22}(\lambda)\Vert \lesssim \varepsilon $ for all $\lambda <0 $ satisfying \eqref{eq:ref_window2_int}. We summarize the last three terms in a linear operator 
\begin{align}
R_{\varepsilon,22}'(\lambda) = (-\varepsilon f_2(\lambda) vKv) |_{G_{2} \to G_{2}} + R_{\varepsilon,22}(\lambda) - S_{\varepsilon,21}^-(\lambda) (S_{\varepsilon,11}^-(\lambda))^{-1}S_{\varepsilon,12}^-(\lambda)  \nonumber 
\end{align}
which satisfies
\begin{align}
\Vert R_{\varepsilon,22}'(\lambda) \Vert  \lesssim |\varepsilon \log \varepsilon| + \varepsilon + |\log\varepsilon|^{-1} \lesssim |\log\varepsilon|^{-1} \label{eq:R_eps22_prime_estimate}
\end{align}
because of \eqref{eq:eps_f2_ref_window2_int}, the estimate $\Vert R_{\varepsilon,22}(\lambda) \Vert \lesssim \varepsilon$ and \eqref{eq:S_eps211112_eps_bound_window2_int}. Then,
\begin{align}
W_{\varepsilon}^-(\lambda) = (1  -  \varepsilon f_1(\lambda) v\Pi_{22}v ) |_{G_{2} \to G_{2}} + R_{\varepsilon,22}'(\lambda). \label{eq:W_eps_reduced} 
\end{align}
Let $Q$ be the orthogonal projection onto $G_{2}$. The space $G_{2}$ is spanned by $Q(v \varphi_1^-)$ and $Q(v \varphi_2^-)$. If we set 
\begin{align}
u_2 = \frac{Q(v \varphi_2^-)}{\Vert Q(v \varphi_2^-) \Vert}, \qquad u_1' = Q(v \varphi_1^-) - \langle u_2, Q(v \varphi_1^-) \rangle  u_2, \qquad u_1 = \frac{u_1'}{\Vert u_1' \Vert }, \nonumber
\end{align}
then $\{ u_1, u_2 \}$ is an orthonormal basis of $G_{2}$. Then,
\begin{align}
\det W_{\varepsilon}^-(\lambda) = \langle u_1 ,  W_{\varepsilon}^-(\lambda) u_1 \rangle \cdot \langle u_2 ,  W_{\varepsilon}^-(\lambda) u_2 \rangle - \langle u_1 ,  W_{\varepsilon}^-(\lambda) u_2 \rangle \cdot \langle u_2 ,  W_{\varepsilon}^-(\lambda) u_1 \rangle \nonumber
\end{align}
and 
$W_{\varepsilon}^-(\lambda)$ is not invertible if and only if $\det W_{\varepsilon}^-(\lambda) = 0$. Now,
\begin{align}
|\langle u_1 ,  W_{\varepsilon}^-(\lambda) u_1 \rangle - 1 |& \lesssim |\log\varepsilon|^{-1}, \label{eq:uWu11} \\
|\langle u_1 ,  W_{\varepsilon}^-(\lambda) u_2 \rangle |& \lesssim |\log\varepsilon|^{-1}, \label{eq:uWu12}\\
|\langle u_2 ,  W_{\varepsilon}^-(\lambda) u_1 \rangle |& \lesssim |\log\varepsilon|^{-1}, \label{eq:uWu21}\\
|\langle u_2 ,  W_{\varepsilon}^-(\lambda) u_2 \rangle - (1 - \varepsilon f_1(\lambda) \Vert Q(v \varphi_2^-) \Vert_{L^2(\R^2)}^2 )  |& \lesssim |\log\varepsilon|^{-1},\label{eq:uWu22}
\end{align}
because of \eqref{eq:R_eps22_prime_estimate} and \eqref{eq:W_eps_reduced}. This shows $\langle u_1 ,  W_{\varepsilon}^-(\lambda) u_1 \rangle \neq 0$ for small enough $\varepsilon$ and $W_\varepsilon^-(\lambda)$ is not invertible if and only if 
\begin{align*}
\langle u_2 ,  W_{\varepsilon}^-(\lambda) u_2 \rangle - \langle u_1 ,  W_{\varepsilon}^-(\lambda) u_2 \rangle \cdot (\langle u_1 ,  W_{\varepsilon}^-(\lambda) u_1 \rangle)^{-1} \cdot \langle u_2 ,  W_{\varepsilon}^-(\lambda) u_1 \rangle = 0. \nonumber
\end{align*}
Let us denote the left hand side by $ u_{\varepsilon}^-(\lambda)$ and let 
\begin{align} \label{eq:nu_n_int_window2}
\nu_n = \Vert Q(v \varphi_2^-) \Vert_{L^2(\R^2)}^2 .
\end{align}
We chose now $C_1$, $C_2$ from the reference window small resp.~large enough such that $\nu_n^{-1} \in (C_1 + \delta, C_2 -\delta)$ for some small $\delta$. We then find for small enough $\varepsilon$ that $u_{\varepsilon}^-(\lambda)>0$ for $\lambda$ such that $\varepsilon f_2(\lambda) \leq  \nu_n^{-1}-\delta$, while $u_{\varepsilon}^-(\lambda')<0$ for $\lambda'$ such that $\varepsilon f_2(\lambda') \geq  \nu_n^{-1}+\delta$. As earlier, we conclude by the intermediate value theorem that there must exist $\lambda_n(\varepsilon)$ with $u_{\varepsilon}^-(\lambda_n(\varepsilon))=0$. Finally, one deduces from the estimates \eqref{eq:uWu11} to \eqref{eq:uWu22} that this $\lambda_n(\varepsilon)$ must satisfy
\begin{align}
|1 - \varepsilon f_1(\lambda_n(\varepsilon))  \nu_n | \lesssim |\log \varepsilon|^{-1} \nonumber
\end{align}
which implies with Lemma 2.8 in \cite{Frank2010} 
\begin{align}
\lambda_n(\varepsilon) = - \frac{\nu_n}{\pi} \frac{\varepsilon}{|\log \varepsilon|} \left(1 + O\left( \frac{\log|\log \varepsilon|}{|\log \varepsilon|} \right) \right) \nonumber
\end{align}
as $\varepsilon \searrow 0$. \\

\textbf{Step 7: Extracting eigenvalue asymptotics - third reference window}\\

In the last step, we consider $\lambda<0$ with 
\begin{align} \label{eq:ref_window3_int}
C_1 \leq \varepsilon f_2(\lambda) \leq C_2
\end{align}
for some $C_1,C_2 >0$. It follows
\begin{align}
\varepsilon f_0(\lambda) &\asymp \varepsilon e^{\frac{1}{\varepsilon}}, \label{eq:eps_f1_ref_window3_int} \\
\varepsilon f_1(\lambda) &\asymp \varepsilon^2 e^{\frac{1}{\varepsilon}}. \label{eq:eps_f2_ref_window3_int}
\end{align}

Let $G_1$ denote the orthogonal complement of $H_{n+1}$ in $H_{n+2}$, i.e.\ $G_1 = H_{n+2}  \ominus H_{n+1}$, which is one-dimensional. We decompose the Schur complement $S_\varepsilon^-(\lambda)$ from \eqref{eq:S_eps_int}, that acts on $H_{n+2}$, now into 
\begin{align}
S_\varepsilon^-(\lambda) = \begin{pmatrix}
S_{\varepsilon,11}^-(\lambda) & S_{\varepsilon,12}^-(\lambda)\\
S_{\varepsilon,21}^-(\lambda) & S_{\varepsilon,22}^-(\lambda)
\end{pmatrix} = \begin{pmatrix}
S_\varepsilon^-(\lambda)|_{H_{n+1} \to H_{n+1}} & S_\varepsilon^-(\lambda)|_{G_1 \to H_{n+1}} \\
S_\varepsilon^-(\lambda)|_{H_{n+1} \to G_1} & S_\varepsilon^-(\lambda)|_{G_1 \to G_1}
\end{pmatrix}. \nonumber
\end{align}
 As before, we find with \eqref{eq:vPv_int_expansion} and \eqref{eq:K_eps122221_comb_int} for $\varepsilon$ small enough,
\begin{align} 
S_{\varepsilon,11}^-(\lambda) &= (1  -  \varepsilon f_0(\lambda) vP_0^-v  - \varepsilon f_1(\lambda) v\Pi_{22}v - \varepsilon f_2(\lambda) vKv) |_{H_{n+1} \to H_{n+1}} + R_{\varepsilon,11}(\lambda) \nonumber \\
S_{\varepsilon,12}^-(\lambda)  & = ( -\varepsilon f_2(\lambda) vKv) |_{G_1 \to H_{n+1}} + R_{\varepsilon,12}(\lambda), \nonumber \\
S_{\varepsilon,21}^-(\lambda) &=   (-\varepsilon f_2(\lambda) vKv)  |_{H_{n+1} \to G_1} + R_{\varepsilon,21}(\lambda)  \nonumber
\end{align}
where $R_{\varepsilon,ij}(\lambda)$ are some linear operators with $\Vert R_{\varepsilon,ij}(\lambda)\Vert \lesssim \varepsilon $ for all $\lambda <0 $ satisfying \eqref{eq:ref_window3_int}. Now, \eqref{eq:eps_f1_ref_window3_int} and \eqref{eq:eps_f2_ref_window3_int} imply 
\begin{align}
\Vert S_{\varepsilon,12}^-(\lambda) \Vert& \lesssim 1 + \varepsilon \lesssim 1, \label{eq:S_eps12_window3_int}\\
\Vert S_{\varepsilon,21}^-(\lambda) \Vert& \lesssim 1 + \varepsilon \lesssim 1. \label{eq:S_eps21_window3_int}
\end{align}
Furthermore, for small enough $\varepsilon$, the operator $S_{\varepsilon,11}^-(\lambda)$ is invertible for all $\lambda$ in the reference window \eqref{eq:ref_window3_int}. This can be seen by applying the SLFG Lemma again to $S_{\varepsilon,11}^-(\lambda)$. One considers $S_{\varepsilon,11}^-(\lambda)$ on $H_n \oplus G_1'$ where $G_1'= H_{n+1} \ominus H_n$ is the orthogonal complement of $H_n$ in $H_{n+1}$. Then, $S_{\varepsilon,11}^-(\lambda)|_{H_{n} \to H_{n}} $ becomes invertible for small enough $\varepsilon$ because $(vP_0^- v) |_{H_{n} \to H_{n}}$ is positive and upon application of the SLFG Lemma, one sees that $S_{\varepsilon,11}^-(\lambda)$ is invertible for small $\varepsilon$ if and only if $\nu_n $ from \eqref{eq:nu_n_int_window2} satisfies $\nu_n = \Vert Q(v\varphi_2) \Vert^2 \neq 0$. But this is true, because $Q(v\varphi_2)\neq 0$ and therefore $\Vert Q(v\varphi_2) \Vert^2> 0$. Thus, $S_{\varepsilon,11}^-(\lambda)$ is indeed invertible for small enough $\varepsilon$. It is also not hard to find that then
\begin{align}
\Vert (S_{\varepsilon,11}^-(\lambda))^{-1} \Vert&  \lesssim \varepsilon^{-2} e^{-\frac{1}{\varepsilon}} . \label{eq:S_eps11inv_window3_int}
\end{align}
By Lemma \ref{lem:slfg_formula}, $S_\varepsilon^-(\lambda)$ is invertible if and only if the Schur complement
\begin{align}
Z_{\varepsilon}^-(\lambda) =  S_{\varepsilon,22}^-(\lambda) - S_{\varepsilon,21}^-(\lambda) (S_{\varepsilon,11}^-(\lambda))^{-1}S_{\varepsilon,12}^-(\lambda) \nonumber
\end{align}
is invertible in $G_1$. Here, 
\begin{align}
\Vert S_{\varepsilon,21}^-(\lambda) (S_{\varepsilon,11}^-(\lambda))^{-1}S_{\varepsilon,12}^-(\lambda) \Vert \lesssim \varepsilon^{-2} e^{-\frac{1}{\varepsilon}} \label{eq:S_eps211112_eps_bound_window3_int}
\end{align}
due to \eqref{eq:S_eps12_window3_int},  \eqref{eq:S_eps21_window3_int} and \eqref{eq:S_eps11inv_window3_int}.  

The operator $Z_{\varepsilon}^-(\lambda)$ acts on the one-dimensional space $G_1$. We can write $Z_{\varepsilon}^-(\lambda)$ as
\begin{align}
Z_{\varepsilon}^-(\lambda) = (1  -   \varepsilon f_2(\lambda) vKv) |_{G_1 \to G_1} + R_{\varepsilon,22}(\lambda) - S_{\varepsilon,21}^-(\lambda) (S_{\varepsilon,11}^-(\lambda))^{-1}S_{\varepsilon,12}^-(\lambda) \label{eq:Z_eps_long}
\end{align}
where $R_{\varepsilon,22}(\lambda)$ is some linear operator with $\Vert R_{\varepsilon,22}(\lambda)\Vert \lesssim \varepsilon $ for all $\lambda <0 $ satisfying \eqref{eq:ref_window2_int}. Let $\tilde{Q}$ be the orthogonal projector onto $G_1$ and let $\tilde{\varphi} = \tilde{Q}(v\varphi_1^-)/ \Vert  \tilde{Q}(v\varphi_1^-) \Vert_{L^2(\R^2)} $, which is a unit vector that spans $G_1$. Then $Z_{\varepsilon}^-(\lambda)$ can be written as 
\begin{align}
Z_{\varepsilon}^-(\lambda) = z_{\varepsilon}^-(\lambda) \tilde{\varphi} \langle \tilde{\varphi}, \, . \, \rangle \nonumber
\end{align}
with 
\begin{align}
z_{\varepsilon}^-(\lambda) = \langle \tilde{\varphi}, Z_{\varepsilon}^-(\lambda) \tilde{\varphi} \rangle \in \R \nonumber
\end{align}
and one sees that $Z_{\varepsilon}^-(\lambda)$ is not invertible on $G_1$ if and only if $z_{\varepsilon}^-(\lambda) = 0$. Similar to the zero flux case, one might then argue that $z_{\varepsilon}^-(\lambda) = 0$ has at least one solution for small enough $\varepsilon$. We denote this solution by $\lambda_{n+1}(\varepsilon)$. Finally, because of $z_{\varepsilon}^-(\lambda_{n+1}(\varepsilon)) = 0$ and \eqref{eq:Z_eps_long}, $\lambda_{n+1}(\varepsilon)$ satisfies
\begin{align}
|1 - \varepsilon f_2(\lambda_{n+1}(\varepsilon)) \langle \tilde{\varphi} , v K v \tilde{\varphi} \rangle  | \lesssim \varepsilon \nonumber
\end{align}
which implies 
\begin{align}
\lambda_n(\varepsilon) = - \exp\left( - \langle \tilde{\varphi} , v K v \tilde{\varphi} \rangle^{-1}  \varepsilon^{-1} \left(1 + O\left(\varepsilon
\right) \right) \right) \nonumber
\end{align}
as $\varepsilon \searrow 0$. 
\begin{flushright}
$\square$
\end{flushright}

\appendix

\section{Radial fields}

We explain how the asymptotic expansions of Theorems \ref{thm:expansion_zero}, \ref{thm:expansion_nonint} and \ref{thm:expansion_int} reduce to those found by Frank, Morozov and Vugalter in \cite{Frank2010} in the special case of radial magnetic field $B$ and radial potential $V$. As before, we denote $v = V^\frac{1}{2}$.

First, in the case of $\alpha = 0$, equation \eqref{eq:thm_zero_flux_asymp} in  Theorem \ref{thm:expansion_zero} gives exactly the same asymptotic expression as is found in Theorem 1.3 of \cite{Frank2010}. There is nothing to be done here.

Next, in the case of $\alpha > 0$, $\alpha \in\mathbb{R} \setminus \mathbb{Z}$, Theorem 1.2 of Frank, Morozov and Vugalter states the eigenvalue expansions
\begin{align}
\lambda_k(\varepsilon) &= - \mu_k \, \varepsilon \, (1 + O(\varepsilon )),  \qquad  k=0, ..., n-2, \label{eq:lambda_k_frank_nonint} \\
\lambda_{n-1}(\varepsilon) &= - \mu_{n-1} \, \varepsilon \, (1 + O(\varepsilon^{ \alpha'} )),  \label{eq:lambda_n-1_frank_nonint}\\ 
\lambda_{n}(\varepsilon) &=  - \mu_n \, \varepsilon^\frac{1}{\alpha'} \, (1+ O(\varepsilon^{\min\{1,\frac{1}{\alpha'}-1\}} )) 
\end{align}
as $\varepsilon \searrow 0$, where
\begin{align}
\mu_k &= \frac{\int_{\R^2} V |\psi_k^-|^2 \dd x}{\int_{\R^2} |\psi_k^-|^2 \dd x },  \qquad  k=0, ..., n-1,  \label{eq:mu_k_frank_nonint} \\
\mu_n &= \left( \frac{4^{\alpha'-1} \Gamma(\alpha') }{\pi \Gamma(1-\alpha')}\, \int_{\R^2} V |\psi_n^-|^2 \dd x \right)^\frac{1}{\alpha'}.\label{eq:mu_n_frank_nonint}
\end{align}
There are two notable differences to the expressions in Theorem \ref{thm:expansion_nonint} of this paper. First, the coefficients $\{ \mu_k \}_{k=0}^n$ are given explicitly in terms of the (generalized) Aharonov-Casher states $\{ \psi_k^- \}_{k=0}^n$ and second, the second order error terms of the eigenvalues $\lambda_k(\varepsilon)$ with linear first order appear to be better in the expansions by Frank, Morozov and Vugalter. Let us explain how the above asymptotic expressions can be derived from Theorem \ref{thm:expansion_nonint} and its proof under the assumption of radial fields.

We first note that by writing the corresponding integrals in polar coordinates that the zero eigenstates $\{ \psi_k^-\}_{k=0}^{n-1}$ are $L^2$-orthogonal to each other when the magnetic field is radial. The projector $P_0^-$ onto the zero eigenspace of $P_-(A)$ can then be represented as
\begin{align}
P_0^- = \sum_{k=0}^{n-1} \frac{1}{\Vert \psi_k^- \Vert_{L^2(\R^2)}^2} \psi_k^- \langle \psi_k^-, \, . \, \rangle \nonumber
\end{align}
which implies
\begin{align}
vP_0^-v = \sum_{k=0}^{n-1} \frac{1}{\Vert \psi_k^- \Vert_{L^2(\R^2)}^2} v\psi_k^- \langle v\psi_k^-, \, . \, \rangle. \label{eq:vPv_radial}
\end{align}
Theorem \ref{thm:expansion_nonint} asserts that the coefficients $\{ \mu_k \}_{k=0}^{n-1}$ to the asymptotically linear eigenvalues $\{ \lambda_k(\varepsilon) \}_{k=0}^{n-1}$ are given by the non-zero eigenvalues of $vP_0^- v$. Assuming now that $V$ is radial, one finds similarly by writing the corresponding integrals in polar coordinates that the states $\{v \psi_k^- \}_{k=0}^n$ are also orthogonal in $L^2(\R^2)$. Given \eqref{eq:vPv_radial}, we conclude that the non-zero eigenvalues of $vP_0^- v$ are
\begin{align}
\mu_k = \frac{\Vert v \psi_k^- \Vert_{L^2(\R^2)}^2}{\Vert \psi_k^- \Vert_{L^2(\R^2)}^2}=\frac{\int_{\R^2} V |\psi_k^-|^2 \dd x}{\int_{\R^2} |\psi_k^-|^2 \dd x },  \qquad  k=0, ..., n-1, \nonumber
\end{align}
which is exactly the expression in \eqref{eq:mu_k_frank_nonint}.

Let us now discuss the coefficient $\mu_n$. It is pointed out in Corollary 5.10 of \cite{Kovarik2022} that in the case of radial magnetic fields, the states $\varphi^-$ and $\psi^-$ mentioned in Theorem \ref{thm:resolvent_non_int} and thus Theorem \ref{thm:expansion_nonint} are given by $\varphi^-= \psi_n^-$ and $\psi^- = d_n^- \psi_{n-1}^-$ where $d_n^-$ is some complex number defined in \cite{Kovarik2022}. This means that $\operatorname{ran}(v P_0^- v)=\operatorname{span}\{ v\psi_k^- \}_{k=0}^{n-1}$ is orthogonal to $v\psi_n^- = v \varphi^-$ and the projector $Q$ from Theorem \ref{thm:expansion_nonint} acts as an identity on $v \varphi^-$. Therefore, by Theorem \ref{thm:expansion_nonint},
\begin{align}
\mu_n &= \left( \frac{4^{\alpha'-1} \Gamma(\alpha') }{\pi \Gamma(1-\alpha')}\, \Vert Q (v \varphi^-) \Vert_{L^2(\R^2)}^{2} \right)^\frac{1}{\alpha'} = \left( \frac{4^{\alpha'-1} \Gamma(\alpha') }{\pi \Gamma(1-\alpha')}\, \int_{\R^2} V |\psi_n^-|^2 \dd x \right)^\frac{1}{\alpha'} .  \nonumber
\end{align}
This shows that Theorem \ref{thm:expansion_nonint} indeed yields \eqref{eq:mu_k_frank_nonint} and \eqref{eq:mu_n_frank_nonint} for the coefficients $\{ \mu_k\}_{k=0}^n$ in case of radial $B$ and $V$.

The pairwise orthogonality of the (generalized) Aharonov-Casher states is also the reason why the second order error terms in \eqref{eq:lambda_k_frank_nonint} and \eqref{eq:lambda_n-1_frank_nonint} are improved. To see this, recall from \eqref{eq:vPv_non_int_expansion} the asymptotic expansion
\begin{align} 
v(P_-(A)-\lambda)^{-1}v&= f_0(\lambda) vP_0^-v +f_1(\lambda) v\psi^- \langle v\psi^-, . \rangle + f_2(\lambda) v\varphi^- \langle v\varphi^-, . \rangle  + O(1) \nonumber
\end{align}
as $\lambda \to 0$, which was the basis of our calculations in the case of non-integer flux. The reason why the second order error term of $\{ \lambda_k(\varepsilon)\}_{k=0}^{n-1}$ is of order $\varepsilon^{\min\{ \alpha',1-\alpha'\}} $ in Theorem \ref{thm:expansion_nonint} is that under the first reference window the terms $f_1(\lambda) v\psi^- \langle v\psi^-, . \rangle $ and $ f_2(\lambda) v\varphi^- \langle v\varphi^-, . \rangle$ above were both treated as general perturbations to $f_0(\lambda) vP_0^-v$ on $H_{n} = \operatorname{span}\{ v\psi_k^- \}_{k=0}^{n-1} $. But with $\varphi^-= \psi_n^-$ and $\psi^- = d_n^- \psi_{n-1}^-$, above expansion becomes
\begin{align} 
v(P_-(A)-\lambda)^{-1}v&= f_0(\lambda) vP_0^-v +f_1(\lambda)|d_n|^2 v\psi_{n-1}^- \langle v\psi_{n-1}^-, . \rangle + f_2(\lambda) v\psi_{n}^- \langle v\psi_{n}^-, . \rangle  + O(1) \nonumber
\end{align}
as $\lambda \to 0$. The orthogonality of the states $\{v \psi_k^- \}_{k=0}^n$ implies that $f_2(\lambda) v\psi_{n}^- \langle v\psi_{n}^-, . \rangle$ does not perturb $f_0(\lambda) vP_0^-v$ on $H_{n}$ and $f_1(\lambda)|d_n|^2 v\psi_{n-1}^- \langle v\psi_{n-1}^-, . \rangle$ only perturbs $f_0(\lambda) vP_0^-v$ along the one-dimensional subspace $\operatorname{span} \{ v\psi_{n-1}^- \} \subset H_n $. Following the rest of the proof and estimating terms more carefully reveals that the second order error terms of $\lambda_k(\varepsilon)$ can indeed be improved to order $\varepsilon$ for $k=0,..., n-2$ and order $\varepsilon^{\alpha'}$ for $k=n-1$ in this case. 

Finally, let us discuss the case of $\alpha > 0$, $\alpha \in \mathbb{Z}$. Theorem 1.3 of Frank, Morozov and Vugalter gives in this case the eigenvalue expansions
\begin{align}
\lambda_k(\varepsilon) &= - \mu_k \, \varepsilon  \left(1+ O\left(\varepsilon |\log \varepsilon| \right) \right),  \qquad k=0, ...,n-1, \label{eq:lambda_k_frank_int}\\
\lambda_{n}(\varepsilon) &=  - \mu_n \frac{\varepsilon}{|\log \varepsilon|} \left(1 + O\left( \frac{\log|\log \varepsilon|}{|\log \varepsilon|} \right) \right), \\
\lambda_{n+1}(\varepsilon) &= - \exp\left( - \mu_{n+1}^{-1}  \varepsilon^{-1} \left(1 + O\left(\varepsilon
\right) \right) \right)  
\end{align}
as $\varepsilon \searrow 0$, where
\begin{align}
\mu_k &= \frac{\int_{\R^2} V |\psi_k^-|^2 \dd x}{\int_{\R^2} |\psi_k^-|^2 \dd x },  \qquad  k=0, ..., n-1,  \label{eq:mu_k_frank_int} \\
\mu_n &= \frac{1 }{\pi }\int_{\R^2} V |\psi_n^-|^2 \dd x,\label{eq:mu_n_frank_int}\\
\mu_{n+1} &= \frac{1 }{4\pi }\int_{\R^2} V |\psi_{n+1}^-|^2 \dd x.\label{eq:mu_nplus1_frank_int}
\end{align}
Again the expressions for the coefficients $\{ \mu_k\}_{k=0}^{n+1}$ look different to those in Theorem \ref{thm:expansion_int} and also the second order error term of the eigenvalues $\{\lambda_k(\varepsilon)\}_{k=0}^{n-1} $ with linear first order term appears to differ. 

The expression \eqref{eq:mu_k_frank_int} is explained as before by pairwise orthogonality of the zero eigenstates $\{\psi_k^- \}_{k=0}^{n-1}$ and pairwise orthogonality of the states $\{v \psi_k^- \}_{k=0}^{n-1}$. Let us explain the remaining coefficients $\mu_n$ and $\mu_{n+1}$. We find in Corollary 6.7 of \cite{Kovarik2022} that in the case of radial magnetic fields, the states $\varphi_1^-$, $\varphi_2^-$ and $\psi^-$ mentioned in Theorem \ref{thm:expansion_int} are given by $\varphi_1^-=\psi_{n+1}^-$, $\varphi_2^-=\psi_n^-$ and $\psi^- = d_n^- \psi_{n-1}^-$. It then follows by pairwise orthogonality of the states $\{v \psi_k^- \}_{k=0}^{n+1}$ that the projections $Q$ resp.~$\tilde{Q}$ of Theorem \ref{thm:expansion_int} act as identities on $v\varphi_2^-$ resp.~$v\varphi_1^-$. By \eqref{eq:mu_n_thm_int}, we find that indeed
\begin{align}
\mu_n = \frac{1}{\pi} \Vert Q(v \varphi_2^-)\Vert^2_{L^2(\R^2)} = \frac{1}{\pi} \Vert v \varphi_2^-\Vert^2_{L^2(\R^2)} = \frac{1 }{\pi }\int_{\R^2} V |\psi_n^-|^2 \dd x. \nonumber
\end{align}
For the coefficient $\mu_{n+1}$, we recall that 
\begin{align}
\mu_{n+1} = \frac{\langle \phi_1 , v K v  \, \phi_1 \rangle}{\Vert \phi_1 \Vert_{L^2(\R^2)}^{2} } \nonumber
\end{align}
where $\phi_1 =  \tilde{Q}  (v \varphi_1^-) =  v \varphi_1^-$ and the operator $K$ is defined in \eqref{eq:definition_K_op}. Since $\kappa = 0$ in case of radial $B$, see again Kova{\v{r}}{\'{\i}}k \cite{Kovarik2022}, the operator $K$ simplifies to 
\begin{align}
K = \frac{1}{4\pi} \Pi_{11}  + \frac{\pi |d_n^-|^2}{4 } \psi^- \langle \psi^- , \, . \, \rangle  \nonumber
\end{align}
and hence $vKv$ becomes
\begin{align}
vKv = \frac{1}{4\pi} v\varphi_1^- \langle v\varphi_1^- , \, . \, \rangle  + \frac{\pi |d_n^-|^2}{4 } v\psi^- \langle v\psi^- , \, . \, \rangle . \nonumber
\end{align}
Finally, because $\phi_1= v\varphi_1^- = v\psi_{n+1}^-$ is orthogonal to $v\psi^- = v\psi_{n-1}^-$, we conclude that
\begin{align}
\mu_{n+1} = \frac{\langle \phi_1 , v K v  \, \phi_1 \rangle}{\Vert \phi_1 \Vert_{L^2(\R^2)}^{2} } = \frac{1}{4\pi} \frac{\Vert v\psi_{n+1}^- \Vert_{L^2(\R^2)}^4}{\Vert v\psi_{n+1}^- \Vert_{L^2(\R^2)}^{2}} = \frac{1 }{4\pi }\int_{\R^2} V |\psi_{n+1}^-|^2 \dd x. \nonumber
\end{align}
This shows that all coefficients $\{ \mu_k\}_{k=0}^{n+1}$ in Theorem \ref{thm:expansion_int} indeed simplify to those of  \eqref{eq:mu_k_frank_int}, \eqref{eq:mu_n_frank_int} and \eqref{eq:mu_nplus1_frank_int} in the case of radial $B$ and $V$.

Similarly to the non-integer flux case, the improved second order error term can be explained by discussing the integer flux resolvent expansion \eqref{eq:vPv_int_expansion} which states that
\begin{align}
v(P_-(A)-\lambda)^{-1}v= f_0(\lambda) vP_0^-v + f_1(\lambda)v\Pi_{22}v +f_2(\lambda) vKv + O(1)  \nonumber
\end{align}
as $\lambda \to 0$. In the proof of Theorem \ref{thm:expansion_int} we arrived at an second order error term of the eigenvalues $\{ \lambda_k(\varepsilon)\}_{k=0}^{n-1}$ of order $|\log \varepsilon|^{-1}$. The reason why a term of this order appeared was that under the first reference window $f_1(\lambda)v\Pi_{22}v$ was treated as a general perturbation to $f_0(\lambda) vP_0^-v$ on $H_n$. Now, when $B$ and $V$ are radial, pairwise orthogonality of the states $\{ v \psi_k^-\}_{k=0}^{n+1}$ implies that $f_1(\lambda)v\Pi_{22}v$ vanishes on $H_n$ and hence does not perturb $f_0(\lambda) vP_0^-v$ on $H_n$. The next largest perturbation to $f_0(\lambda) vP_0^-v$ on $H_n$ that remains comes from the term $f_2(\lambda) vKv$ (to be specific, the summand of $K$ where $\psi^-=d_n^- \psi_{n-1}^-$ appears) which only perturbs $f_0(\lambda) vP_0^-v$ along the one-dimensional subspace $\operatorname{span} \{ v\psi_{n-1}^- \} \subset H_n $. Carefully finishing the proof with this additional information shows that the eigenvalue expansions of $\{ \lambda_k(\varepsilon)\}_{k=0}^{n-1}$ can be improved to
\begin{align}
\lambda_k(\varepsilon) &= - \mu_k \, \varepsilon  \left(1+ O(\varepsilon ) \right),  \qquad k=0, ...,n-2, \nonumber \\
\lambda_{n-1}(\varepsilon) &= - \mu_{n-1} \, \varepsilon  \left(1+ O\left(\varepsilon |\log \varepsilon| \right) \right) \nonumber
\end{align}
as $\varepsilon \searrow 0$. We see that the order of the error term of the $(n-1)$-th eigenvalue coincides with that given by \eqref{eq:lambda_k_frank_int}. For the eigenvalues $\lambda_0(\varepsilon)$, ..., $\lambda_{n-2}(\varepsilon)$, we gain a slightly improved error term.

\section*{Acknowledgements}

The author is immensely grateful to Hynek Kova{\v{r}}{\'{\i}}k for the introduction to and helpful discussions on this problem. He also gratefully acknowledges the hospitality at Università degli studi di Brescia during his research stay in February-March 2024. The author expresses his gratitude towards Timo Weidl for suggestions on improvements and many helpful comments on exposition. He is also grateful to Rupert Frank and Tobias Ehring for helpful remarks.

\section*{Declarations}

Funding: No funding was received to assist with the preparation of this manuscript. 

Competing interests: The author has no competing interests to declare.

\printbibliography

\end{document}